\documentclass[10pt,draft]{article}
\usepackage[top=1.25in,left=1in,right=1in,bottom=1.25in]{geometry}  
\usepackage[utf8]{inputenc}
\usepackage[T1]{fontenc}
\usepackage{mathrsfs,amssymb,dsfont}
 \usepackage{theorem}
\usepackage{microtype}
\usepackage{verbatim}
\usepackage[intlimits]{empheq}
\usepackage{latexsym}
\usepackage[round]{natbib}
\usepackage[format=hang,
 font ={footnotesize,sf},
 labelfont = {bf},
 margin =1cm,
 aboveskip = 5pt,
 position = bottom]{caption}

\newcommand{\A}{\mathscr{E}}
\newcommand{\R}{\mathbb{R}}

\newcommand{\C}{\mathbb{C}}
\newcommand{\E}{\mathbb{E}}
\newcommand{\N}{\mathbb{N}}

\newcommand{\B}{\mathbb{B}}

\newcommand{\Ss}{\mathcal{S}}

\newcommand{\LLl}{\mathcal{L}}
\newcommand{\WW}{\mathcal{W}}
\newcommand{\D}{\mathcal{D}}
\newcommand{\HH}{\mathcal{H}}
\newcommand{\F}{\mathcal{F}}

\newcommand{\1}{\mathds{1}}
\newcommand{\gvr}{\varrho}
\newcommand{\gvrp}{{\varrho/p}}

\DeclareMathOperator{\Id}{id}
\newcommand{\fclass}[1]{\F_{\text{#1}}}
\newcommand{\sord}[1]{\le_{\text{#1}}}
\newcommand{\fsord}{\le_{\F}}

\newcommand\llambda{{\mathchoice
 {\lambda\mkern-4.5mu{\raisebox{.4ex}{\scriptsize$\backslash$}}}
 {\lambda\mkern-4.83mu{\raisebox{.4ex}{\scriptsize$\backslash$}}}
 {\lambda\mkern-4.5mu{\raisebox{.2ex}{\footnotesize$\scriptscriptstyle\backslash$}}}
 {\lambda\mkern-5.0mu{\raisebox{.2ex}{\tiny$\scriptscriptstyle\backslash$}}}}} 

\def\stackrellow#1#2{\mathrel{\mathop{#1}\limits_{#2}}}

\newcommand{\abs}[1]{\left| #1 \right|}
\newcommand{\nnorm}[1]{\left\| #1 \right\|}

\renewcommand{\theequation}{\thesection.\arabic{equation}}
\makeatletter

\renewcommand{\p@enumi}{\thetheorem-}
\makeatother

\usepackage{theorem} 

\newtheorem{theorem}{Theorem}[section]

\newtheorem{propos}[theorem]{Proposition}

\newtheorem{definition}[theorem]{Definition}

\newtheorem{cor}[theorem]{Corollary}

\newtheorem{lemma}[theorem]{Lemma}

\newtheorem{example}{Example}[section]

{\theorembodyfont{\rmfamily}\newtheorem{Example-rm}[theorem]{Example}}
{\theorembodyfont{\rmfamily}\newtheorem{Examples-rm}[theorem]{Examples}}

\newtheorem{Remark}[theorem]{Remark}

\newtheorem{rem}{Remark}[section]

{\theorembodyfont{\rmfamily}\newtheorem{Remark-rm}[theorem]{Remark}}

\newcommand{\BOX}{\mbox{{\ensuremath{\Box}}\hspace{-0.5mm}}}

\newenvironment{proof}{{\par\noindent\bf Proof:}}{\mbox{}\hfill$\BOX$\\}

{\noindent\textbf{Proof of claim \protect\ref{#1}:}} {\mbox{}\hspace*{\fill}\\}

{\par\bigskip \noindent\textbf{Proof of Theorem \protect\ref{#1}:}} 
{\mbox{}\hspace*{\fill}$\BOX$\\}
{{\par\bigskip \noindent\bf{Proof of Theorems \protect\ref{#1}--\protect\ref{#2}:}}}
{\mbox{}\hfill\hspace*{\fill}$\BOX$\\}

\newcommand{\bee}{\begin{equation}} 
\newcommand{\ene}{\end{equation}}

\newcommand{\beE}{\begin{Equation}} 
\newcommand{\enE}{\end{Equation}}

\newcommand{\bdi}{\begin{displaymath}}  
\newcommand{\edi}{\end{displaymath}}

\newcommand{\bDI}{\begin{Displaymath}}  
\newcommand{\eDI}{\end{Displaymath}}

\newcommand{\bqa}{\begin{eqnarray}} 
\newcommand{\eqa}{\end{eqnarray}}

\newcommand{\bqA}{\begin{Eqnarray}} 
\newcommand{\eqA}{\end{Eqnarray}}

\newcommand{\bea}{\begin{eqnarray*}}  
\newcommand{\ena}{\end{eqnarray*}}

\newcommand{\beA}{\begin{Eqnarray*}}  
\newcommand{\enA}{\end{Eqnarray*}}

\theoremstyle{break}

\usepackage[usenames, dvipsnames]{xcolor}

\title{Comparison of time-inhomogeneous Markov processes}
\author{Ludger R\"uschendorf*, Alexander Schnurr**, Viktor Wolf*}

\allowdisplaybreaks

\begin{document}

\maketitle\noindent
\begin{normalsize}\textit{\begin{center}
*Department of Mathematical Stochastics, University of Freiburg, Eckerstra\ss e~1, 79104 Freiburg, Germany\\
**Faculty of Mathematics, TU Dortmund, Vogelpothsweg 87, 44227 Dortmund, Germany \end{center}}
\end{normalsize}
\subsection*{Abstract}
Comparison results are given for time-inhomogeneous Markov processes with respect to function classes induced stochastic orderings. 
The main result states comparison of two processes, provided that the comparability of their infinitesimal generators as well as an invariance property
of one process is assumed. The corresponding proof is based on a representation result for the solutions of inhomogeneous 
evolution problems in Banach spaces, which extends previously known results from the literature. Based on this representation, an ordering result for Markov processes induced by bounded and unbounded function classes is established. We give various applications to time-inhomogeneous diffusions, to processes with independent increments and to L\'{e}vy driven diffusion processes. 
\\

\noindent
\textit{AMS subject}: 60E15; 60J35; 60J75

\noindent
\textit{Keywords}: Stochastic ordering, Markov processes, Infinitesimal generators, Processes with independent increments, Evolution systems


\setcounter{equation}{0}
\section{Introduction}\label{sec-Intro}
Stochastic ordering and comparison results for stochastic models are topics which have undergone an intensive development in various areas 
of probability and statistics such as decision theory, financial economics, insurance mathematics, risk management, queueing theory and many
others. There are several approaches for comparing homogeneous Markov processes in the literature. An approach which investigates the 
infinitesimal generator of Markov processes in order to derive comparison results was established by \cite{Massey-1987}. 
A diffusion equation approach is given for the study of stochastic monotonicity in \cite{Herbst-Pitt-1991}. \cite{Bassan-Scarsini-1991} consider partial orderings for stochastic processes induced by expectations of convex or increasing convex (concave or increasing concave) functionals. 
For bounded generators and in the case of discrete state spaces \cite{Daduna-Szekli-2006} give a comparison result for the stochastic ordering of Markov processes in terms of their generators. \cite{Ruesch-2008} established a comparison result for homogeneous Markov processes using boundedness conditions on the order defining function classes. 
Comparison results for homogeneous Markov processes with transition functions defined on general Banach spaces are given in \cite{Ruesch-Wolf-2010}. The results are based on an integral representation of solutions to the inhomogeneous Cauchy problem. The directionally convex ordering of a 
special system of time-inhomogeneous interacting diffusions was considered in a similar way in \cite{Cox-etal-1996} and \cite{Greven-etal-2002}.

Mainly motivated by financial applications a stochastic analysis approach has been developed in \cite{Karoui-etal-1998}, \cite{Bellamy-Jean-2000}, \cite{Gushchin-Mordecki-2002} and \cite{Bergen-Ruesch-2006,Bergen-Ruesch-2007}. In these papers comparison results for $d$-dimensional semimartingales are established based on the It\^{o}-formula and on the Kolmogorov backward equation (see also \cite{Guendouzi-2009}). 
A coupling approach for diffusion processes and stochastic volatility models has been developed in \cite{Hobson-1998}. An approximation method is used in \cite{Bergen-Rueschen-2007} to give some comparison results for L\'{e}vy processes and processes with independent increments. Several examples and applications to $\alpha$-stable processes, \emph{NIG} processes and \emph{GH} processes are discussed. For multidimensional L\'{e}vy processes using an analytical formula \cite{Bauerle-etal-2008} investigate dependence properties and establish some comparison results for the supermodular 
order. They study the question, whether dependence properties and orderings of the distributions of a L\'{e}vy process can be characterized by corresponding properties of the L\'{e}vy copula.

Comparison results for time-inhomogeneous Markov processes based on the theory of evolution systems on general Banach spaces as used in this paper have not been investigated before. For our main comparison result we establish a representation result for solutions of the evolution problem associated with a family of infinitesimal generators. 
We do not use for our approach approximation arguments or coupling arguments as in \cite{Hobson-1998}, \cite{Greven-etal-2002} or in \cite{Bergen-Rueschen-2007}. Moreover, the application of the theory of evolution systems on Banach spaces allows us to reduce regularity assumptions necessary in the stochastic analysis approach based on It\^{o}'s formula.

Applications of this comparison result are given to processes with independent increments (in the sequel abbreviated as PII), inhomogeneous diffusions and to diffusion models driven by L\'{e}vy processes. Therefore, we introduce generators of several interesting orderings for which the conditions of the comparison result do hold. 
Since we are interested in comparison of Markov processes $(X_t)_{t\ge0}$ and $(Y_t)_{t\ge0}$ w.r.t. an integral stochastic order $\fsord$, that is $Ef(X_t) \le Ef(Y_t),\,t\ge 0$ for all $f\in\F$, an integrability condition like 
\begin{equation}\label{eq-integrability condition}
 f\in\bigcap_{t\ge0}\LLl^1(P^{X_t})\cap\LLl^1(P^{Y_t}) \text{ for all } f\in\F
\end{equation}
is indispensable and is made from now on. Also for general state spaces $\E$ let $\F$ be a set of real functions on $\E$ in some Banach function space $\B$ and let $\fsord$ denote the corresponding stochastic order on $\mathcal M^1(\E,\A)$, the set of probability measures on $\E$, defined by 
\begin{equation}\label{eq-definition of stochastic ordering}
 \mu \fsord \nu \text{ if }\int f d\mu \le \int f d\nu,
\end{equation}
for all $f\in\F$ such that the integrals exist. Some interesting examples of stochastic orderings $\fsord$ are given by the following function classes $\F$ for $\E=\R^d$, 
\begin{align}
 \fclass{st}& := \{f:\R^d\to\R; f\text{ is increasing}\}\label{eq-fst def}\\
 \fclass{cx}& := \{f:\R^d\to\R; f\text{ is convex}\}\label{eq-fcx def}\\
 \fclass{dcx}& := \{f:\R^d\to\R; f\text{ is directionally convex}\}\label{eq-fdcx def}\\
 \fclass{sm}& := \{f:\R^d\to\R; f\text{ is supermodular}\}\label{eq-fsm def}\\
 \fclass{icx}& := \fclass{cx}\cap \fclass{st},\;\fclass{idcx}=\fclass{dcx}\cap \fclass{st},\;\fclass{ism}=\fclass{sm}\cap \fclass{st}.\label{eq-mixes def}
\end{align}
or by subclasses of them. Orderings induced by one of these function classes $\F$ are also generated by $\F\cap C^\infty$, where $C^\infty$ is the set of all infinitely differentiable functions, as well as by many other order generating function classes $\F^0\subset\F$. For notions, properties and applications on stochastic orders we refer to \cite{Tong-1980}, \cite{Shaked-Shanthikumar-1994} and \cite{Mueller-Stoyan-2002}.

As mentioned before the overall aim is to compare Markov processes w.r.t. stochastic orderings $\fsord$. For stochastic orderings induced by bounded function classes the Banach space $B(\R^d)$ of bounded functions respectively $\LLl^\infty(\nu)$ of measurable and essentially bounded functions, where $\nu$ is a suitable measure, respectively $C_\text{b}(\R^d)$ is utilized. 
For processes with translation invariant transition functions comparison results are given for unbounded function classes on modified $\LLl^p$-spaces, which are introduced for this purpose in Section \ref{sec-Evolution systems on BS}. We establish easy to verify and flexible ordering criteria which allow us to apply these results to general classes of models as for example to PII and to L\'{e}vy driven diffusion processes and to general order defining function classes. 
L\'{e}vy driven diffusions, PII, stochastic interest rate models and stochastic volatility models driven by a L\'{e}vy process have found considerable recent attention in the financial mathematics literature. Using the general frame of evolution system and Banach function space theory we are able to treat such classes of time-inhomogeneous Markov processes.

In detail our paper is structured as follows: In Section \ref{sec-Evolution systems on BS} we introduce some relevant notions and results from evolution system theory on general Banach spaces. We establish that transition functions associated to inhomogeneous Markov processes are evolution systems on certain Banach spaces. In particular we introduce $\LLl^p$-type spaces with modified $\LLl^p$-norm and prove that the transition operators define evolution systems on these modified $\LLl^p$-spaces.
As a consequence several interesting examples like comparison w.r.t. convex functions can be dealt with in generality. In Section \ref{sec-Comp of MP} we introduce the weak evolution problem. The main tool of our ordering method is the representation theorem for solutions of the weak evolution problem given in Theorem \ref{theo-evolution problem}. This result extends corresponding results previously known in the literature. 
As consequence we obtain a general comparison result for inhomogeneous Markov processes. 

In Section \ref{sec-application} we discuss several applications that can be dealt with by the generalized approach in this paper. We apply this approach to L\'{e}vy driven diffusion on $C_0(\R^d)$ in Section \ref{subsec-Apps and Exs-LDDP}, to PII's in Section \ref{subsec-Apps and Exs-PII} and to L\'{e}vy driven diffusions on the unbounded function class $\overline{\mathcal L}_2^2(\nu)$ in Section \ref{subsec-Levy driven diffusion on L^2(nu)}. 
In Section \ref{sec-general remark} we consider as particular example the Sobolev Slobodeckii spaces $\HH^r(\R^d)$. The transition operators then define pseudo-differential operators on $\HH^r(\R^d)$. 
In this case our representation result is closely related to a representation result in \cite{Boettcher-2008}. 

\setcounter{equation}{0}
\section{Evolution systems and their infinitesimal generators}%
\label{sec-Evolution systems on BS}
In this section we at first recollect some notions and results from evolution system theory which is strongly related to the semigroup theory on general Banach spaces. Our main reference is \cite{Friedman-1969}; see also \cite{Pazy-1983} and \cite{Engel-Nagel-2000} for examples and applications.

A two parameter family of bounded linear operators $(T_{s,t})_{s\le t},\, s,t\in \R_+$ on a Banach space $(\B,\|\cdot\|)$ is called an \emph{evolution system} $({ES})$ if the following three conditions are satisfied:
\begin{enumerate}
 \item $T_{s,s}=\Id_\B=\1,$
 \item $T_{s,t}\1=\1$,
 \item $T_{s,t} = T_{s,u}T_{u,t}$\text{ for }$0\le s\le u \le t$.\qquad\text{ (\emph{evolution property})}\label{eq-enum-evo property}
\end{enumerate}
The evolution system is strongly continuous if the map $(s,t)\mapsto T_{s,t}$ is \emph{strongly continuous} for $0\le s\le t$. $T$ is called a \emph{contraction}, if
\begin{equation}\label{eq-contraction}
 \|T_{s,t}f\|\le \|f\|\text{ for }0\le s\le t,\,f\in\B.
\end{equation}
For a strongly continuous ${ES}$ we consider the corresponding family of right derivatives or \emph{infinitesimal generators} $A_s:\D(A_s)\subset\B\to\B$
\begin{equation}\label{eq-infi gen}
 A_sf = \lim_{h\downarrow 0}\frac{1}{h}\Bigl(T_{s,s+h}f - f\Bigr)\text{ for }s>0
\end{equation}
defined on its \emph{domain}
\begin{equation}\label{eq-domain}
 \D(A_s):=\Biggl\{f\in \B\Big|\lim_{h\downarrow 0}\frac{1}{h}\Bigl(T_{s,s+h}f - f\Bigr)\text{ exists}\Biggr\}.
\end{equation}

For the reader's convenience we give a proof for the following elementary result (cf. \cite{guli-vancarst}, Section 2.3). The corresponding results for semigroups on Banach spaces can be found in \cite{Dynkin-1965} or \cite{Friedman-1969}.  Integrals on Banach spaces are meant in Riemann sense (cf. e.g. \cite{ethierkurtz}, page 8 and Lemma 1.4).

\begin{lemma}\label{lem-basic lemma for evo system}
 Let $(A_s,\D(A_s))_{s>0}$ be the family of infinitesimal generators of a strongly continuous evolution system $T=(T_{s,t})_{s\le t}$. Then, it holds:
\begin{enumerate} 
 \item For $f\in\B$ and $0\le s\le u\le t$,    
 \begin{equation}\label{eq-lem- first continuity prop for T}
 \lim_{h\downarrow 0}\frac{1}{h}\int_t^{t+h}T_{s,u}f \,du = T_{s,t}f.
 \end{equation}
 \item If $s\mapsto T_{s,t}f$ is right-differentiable for $0<s<t$, then it holds 
 \begin{equation}\label{eq-ex-backward equation}
 \frac{d^+}{ds}T_{s,t}f = - A_sT_{s,t}f. \qquad\text{ (backward equation)}
 \end{equation}
In particular $T_{s,t}f\in \D(A_s)$ for $0 <s<t$. 
 \item If $f\in\D(A_s)$ for $s,t\in\R_+,\,s<t$ it holds that
 \begin{equation}\label{eq-ex-forward equation}
 \frac{d^+}{dt}T_{s,t}f = T_{s,t}A_tf.\qquad\text{ (forward equation)}
 \end{equation}
\end{enumerate}
\end{lemma}

\begin{proof} Part (1) follows directly from the continuity of $t\mapsto T_{s,t}f$ for all $s\in[0,t]$ and $f\in\B$. For (2) let $s < t$, then the evolution property and the strong continuity of $T$ lead to 
\begin{align*}
 \frac{d^+}{ds}T_{s,t} f &= \lim_{h\downarrow 0} T_{s,s+h} \frac{d^+}{ds}T_{s,t} f\\
  &= \lim_{h\downarrow 0}  T_{s,s+h} \left( \frac{d^+}{ds}T_{s,t}f- \frac{T_{s+h,t}f-T_{s,t} f}{h} \right) 
	  + \lim_{h\downarrow 0} T_{s,s+h} \frac{T_{s+h,t} f-T_{s,t} f}{h} \\
  &=0 + \lim_{h\downarrow 0} \frac{T_{s,t}f-T_{s,s+h}T_{s,t} f}{h}\\
  &=-\lim_{h\downarrow 0} \frac{T_{s,s+h} - \Id_{\B}}{h} T_{s,t} f\\	
  &= -A_sT_{s,t} f.
\end{align*}
Now let $f\in\D(A_s),\,s\in[0,t]$ then we obtain again due to the evolution property
\begin{align*}
 \frac{d^+}{dt}T_{s,t}f& = \lim_{h\downarrow 0}\frac{1}{h}\bigl(T_{s,t+h}f - T_{s,t}f\bigr)\\
 & = \lim_{h\downarrow 0}T_{s,t}\Bigl(\frac{1}{h}\bigl(T_{t,t+h}f - f\bigr)\Bigr)\\
 & = T_{s,t}\Bigl(\lim_{h\downarrow 0}\frac{1}{h}\bigl(T_{t,t+h}f -f\bigr)\Bigr)\\
 & = T_{s,t}A_tf
\end{align*}
and part (3) is done. 
\end{proof}

For a time-inhomogeneous Markov process $(X_t)$ on a measure space $(\E,\A)$ let $(P_{s,t})_{s\le t}$ denote the transition kernel or \emph{transition function}
\begin{equation}\label{eq-transition kernel}
P_{s,t}(x,V) = P(X_t\in V\mid X_s = x),\,x\in\E,\,V\in\A 
\end{equation}
and $T=(T_{s,t})_{s\le t}$ the corresponding evolution system, i.e. the transition operator
\begin{align}
 T_{s,t}f(x) & =  \int_\E P_{s,t}(x,dy) f(y)\label{eq-evolution via trans func}\\
 & =  E(f(X_t)\mid X_s = x)\label{eq-evolution system via cond. expectation}
\end{align}
for $f$ in a suitable Banach space $\B$ of functions on $\E$. This puts Markov processes in the framework of evolution systems and evolution equations. 

As first example we consider the Banach spaces 
 $
 \LLl^\infty(\nu):=\{f:\R^d\to \bar \R \mid f \text{ is }\nu\text{-measurable, }\|f\|_\infty < \infty \}
 $
of $\nu$-measurable and essentially bounded functions, where $\nu$ is a suitable measure on $\R^d$, $\,B(\R^d)$ the class of bounded functions on $\R^d$ and $C_{\text{b}}(\R^d)$ the continuous functions in $B(\R^d)$. Recall that for Markov processes the properties (1),(2) and (3) from the definition of $ES$ always hold true on these Banach spaces. Moreover, observe that uniform continuity of the transition kernels in \eqref{eq-evolution via trans func} implies strong continuity of the transition operator $T$:
\begin{equation}\label{eq:continuity condition for trans kernel}
 \begin{split}
 \sup_{x\in\R^d} & \int_{\R^d}|P_{s,t}(x,dy) - P_{s,u}(x,dy)| \to 0\text{ as }t\to u\text{, and}\\
 \sup_{x\in\R^d} & \int_{\R^d}|P_{s,t}(x,dy) - P_{u,t}(x,dy)| \to 0\text{ as }s\to u.
 \end{split}
\end{equation}

\begin{propos}\textbf{\emph{(ES-property for time-inhomogeneous Markov processes on bounded function classes)}} \label{prop22} \\
Let $(P_{s,t})_{s\le t}$ be the transition kernel of a time-inhomogeneous Markov process $(X_t)_{t\ge 0}$ on $(\E,\A)$. If the transition kernel $(P_{s,t})_{s\le t}$ satisfies
condition \eqref{eq:continuity condition for trans kernel} then the operator defined in \eqref{eq-evolution via trans func} is a strongly continuous contraction $ES$ on $\B \in \{B(\R^d), \LLl^\infty(\nu)\}$. If additionally, $f\in C_\text{b}(\R^d)$ implies $T_{s,t}f\in C_\text{b}(\R^d)$, then the operator is a strongly continuous contraction $ES$ on $\B = C_\text{b}(\R^d)$.
\end{propos}

\begin{proof}
 Since $T$ is a family of bounded operators it leaves bounded functions invariant, that is, for $f\in C_\text{b}(\R^d)$ we obtain
 \[
 \| T_{s,t} f\|_\infty \le \sup_x \int_{\R^d} P_{s,t}(x,dy) |f(y)| \le \|f\|_\infty.
 \]
Hence, the strong continuity follows due to
 \[
 \| T_{s,t} f - f \| = \sup_x | T_{s,t} f(x) - T_{s,s} f(x) | 
		 \le \|f\|_\infty \cdot \sup_x \int_{\R^d} | P_{s,t}(x,dy) - P_{s,s}(x,dy) | 
 \]
for all $s\le t$ and the continuity property \eqref{eq:continuity condition for trans kernel}. \newline 
The proof for $\B \in \{B(\R^d), \LLl^\infty(\nu)\}$ is very similar and therefore omitted.
\end{proof}

\begin{Remark}\label{rem-For PII C_b remains invariant}
\begin{description}
 \item[(a)] For a PII $L = (L_t)_{t\ge 0}$ the transition operators on $\B$ are given by
 \begin{equation}\label{eq-rem-valueable equation for PII}
 T_{s,t}f(x) = \int_{\R^d}P^{L_t-L_s}(dy)f(x + y).
 \end{equation}
If $\B = C_b(\R^d)$ the map $x\mapsto T_{s,t}f(x)$ is continuous since the continuity of $f$ transfers to $(T_{s,t})$ thus $C_{\text{b}}(\R^d)$ is invariant under $(T_{s,t})_{s\le t}$. In particular the transition kernels of L\'{e}vy process such as Brownian motion, NIG, VG, GH processes have this invariance property.
 \item[(b)] Note that the Brownian motion respectively its associated Brownian kernel
 \[
	 P_{s,t}(x,dy) = \int_{\R^d}\frac{1}{\sqrt{2\pi (t-s)}} \exp\bigg(-\frac{|x-y|^2}{2(t-s)}\bigg) dy
	 \]
 for $x\in\R^d$ and $s\le t$ is not a strongly continuous $ES$ on $C_\text b(\R^d)$. Hence, it is vital to consider further Banach spaces with different norms to establish the $ES$-property for Markov processes. 
In Proposition \ref{propos-parade Beispiel for evo systems} we consider the Banach space of $p$-integrable functions on $\R^d$ and show that in case of translation invariant transition function the $ES$-property does hold.
 \item[(c)] In Section \ref{subsec-Apps and Exs-LDDP} we will see that L\'{e}vy driven diffusion defined via the stochastic differential equation \eqref{eq-levy driven processes as SDGL} possesses the $ES$-property on $C_0$. Thus, in order to establish stochastic ordering results induced by bounded function classes the Banach space $(C_0,\|\cdot\|_\infty)$ is used as a reference space.
\end{description}
\end{Remark}

For many stochastic orderings the function spaces $C_{\text{b}}(\R^d)$, $B(\R^d)$ or $\LLl^\infty(\nu)$ are sufficient; some orderings as for example the convex ordering $\sord{cx}$ however do not allow bounded generating classes of the ordering. 
The following proposition shows that Markov processes with translation invariant transition functions $(P_{s,t})_{s\le t}$ are strongly continuous ES on $\LLl^p$-spaces. Hereto, we consider the $p$-integrable functions on $\R^d$, which we denote by
 \[
 \LLl^p(\R^d) = \{f:\R^d\to \bar \R| f \text{ is measurable, }\|f\|_p < \infty \},\,1\le p <\infty.
 \]

\begin{propos}\textbf{\emph{(ES-property for time-inhomogeneous translation invariant Markov processes on $\LLl^p(\R^d)$)}}\label{propos-parade Beispiel for evo systems} \\
Let $(X_t)_{t\ge 0}$ be a time-inhomogeneous Markov process on $\R^d$ with translation invariant transition function $(P_{s,t})_{s\le t}$. Then the family of transition operators $T=(T_{s,t})_{s\le t}$ defined by
 \begin{equation}\label{eq-propos-expactation operator for tranlation inv MP}
 T_{s,t}f(x) = \int_{\R^d} P_{s,t}(x,dy)f(y) 
		 = \int_{\R^d}P_{s,t}(0,dy)f(y+x)
 \end{equation}
 for $f\in\LLl^p(\R^d)$, is a strongly continuous contraction ES on $\LLl^p(\R^d),\,1\le p$.
\end{propos}

\begin{proof}
 Let $s\le t$ and set $P_{s,t}(0,dy) =:P_{s,t}(dy)$. We have for $f\in \LLl^p(\R^d)$
\begin{align*}
 \|T_{s,t}f - f\|_p^p & =  \int_{\R^d} \Bigl|\int_{\R^d}P_{s,t}(dy)f(x+y) - f(x) \Bigr|^p dx\\
& \le  \int_{\R^d} P_{s,t}(dy)\int_{\R^d}|f(x+y) - f(x)|^p dx
 =: \int_{\R^d} P_{s,t}(dy) h(y),
\end{align*}
where $h$ is a bounded and continuous function with $h(0)=0$ (see \citet[E.~34.10]{Sato-1999}). For each $\varepsilon>0$ we can find a $\delta>0$ such that $\int_{\{|y|\le \delta\}} P_{s,t}(dy)h(y) < \varepsilon$. 
By the triangle inequality to show strong continuity of $T_{s.t}$ it is enough to consider for $u\in \R_+$ the case that $s\le t$ and $s,t\to u$. Then 
\begin{align*}
\lim_{\substack{s\to u, t\to u,\\s<t}} \int_{\R^d}P_{s,t}(dy)h(y) 
&\le \varepsilon + \lim_{s\to u,t\to u} \int_{\{|y|>\delta\}}P_{s,t}(dy)h(y)\\
&= \varepsilon + \int_{\{|y|>\delta\}}P_{u,u}(dy)h(y)\\
&= \varepsilon + \int_{\{|y|>\delta\}}\delta_0(dy)h(y)  = \varepsilon.
\end{align*}
Finally, we obtain due to the convexity of $x\mapsto |x|^p$,
\begin{align}
 \|T_{s,t}f\|_p^p & = \int\Bigl|\int P_{s,t}(dy)f(x+y)\Bigr|^pdx\notag\\
& \le \int P_{s,t}(dy)\int|f(x+y)|^pdx\label{eq-propos-inequality for contraction prop}\\
 & = \int P_{s,t}(dy)\int|f(x)|^pdx=\|f\|_p^p\notag.
\end{align}
Thus, the operator norm of $T_{s,t}$ is bounded by one, i.e. $\|{T_{s,t}}\|\le 1$ and $T_{s,t}f\in\LLl^p(\R^d)$ for all $0\le s\le t$, i.e. $T$ is a strongly continuous contraction evolution system on $\LLl^p(\R^d)$.
\end{proof}

For some classes of models it is possible to establish the $ES$-property by comparison to the translation invariant case.

\begin{cor}\label{cor-trans operator is evo if it is bounded by transl inv op}
Let $(X_t)_{t\ge 0}$ be a time-inhomogeneous Markov process on $\R^d$ with transition function $(P_{s,t})_{s\le t}$ and transition operator $T=(T_{s,t})_{s\le t}$ defined by \eqref{eq-evolution via trans func}. 
If there exists a positive constant $c$ and a translation invariant transition function $(Q_{s,t})_{s\le t}$ such that
\begin{equation}
\int_{\R^d}f(y) P_{s,t}(x,dy)\le c\int_{\R^d}f(y) Q_{s,t}(x,dy), 
\end{equation}
for all $x\in\R^d$, all positive functions $f$ on $\R^d$ and all $s,\,t\in\R_+$ with $s \le t$, then $T$ is a strongly continuous ES on $\LLl^p(\R^d),\,1\le p < \infty$.
\end{cor}

As we see in Proposition \ref{propos-parade Beispiel for evo systems} the translation invariance property of the Lebesgue measure and the translation invariance of the associated transition function is crucial for $T$ to be a strongly continuous contraction $ES$ on $\LLl^p(\R^d)$. 
For a $\sigma$-finite measure $\nu$ we circumvent the lack of the invariance property by introducing a suitable weighted $\sup$-norm on a sufficiently large subspace of $\LLl^p(\nu)$:
 \[
 \bar{\LLl}^p_\varrho(\nu):=\Bigl\{f\in\LLl^p(\nu)\,\big|\;\|f\|_{p,\gvr}^\ast 
:=\sup_{y\in\R^d}\frac{1}{(1 + \|y\|)^\frac{\gvr}{p}} 
\Bigl(\,\int_{\R^d} \!|f(x + y)|^p d\nu(x)\Bigr)^{\frac{1}{p}}<\infty\Bigr\}
 \]
for $1\le p<\infty$ and $\gvr \ge 0$. For $p =\infty$ we define $\|\cdot\|_{\infty,\gvr} = \| \cdot\|_\infty = \|\cdot\|_{\infty,\nu}$. 
Thus, the space $\bar{\LLl}_\varrho^\infty(\nu)$ equals $\LLl^\infty(\nu)$. The $\LLl^p$-type space $( \bar{\LLl}^p_\varrho(\nu), \|\cdot\|_{p,\gvr}^\ast),\gvr \ge 0,\,1\le p \le \infty$ is a Banach space.

\begin{lemma}\label{lem-modified L p space is a BS}
 $(\bar{\LLl}^p_\gvr(\nu),\|\cdot\|_{p,\gvr}^\ast),\,\gvr\ge 0,\,1\le p \le \infty$ is a Banach space.
\end{lemma}

\begin{proof}
 By definition $\|\cdot\|_{p,\gvr}^\ast$ is a norm and therefore, $\bar{\LLl}^p_\gvr(\nu)$ is a vector space. It remains to show that $\bar{\LLl}^p_\gvr(\nu)$ is complete. This is done similarly to the proof of the \emph{Riesz--Fischer} Theorem.

Let $(f_j)_{j\in\N}$ be a Cauchy sequence in $\bar{\LLl}^p_\gvr(\nu)$, that is 
\begin{equation}\label{eq-lem-cauchy sequence definition}
 \|f_j - f_k\|_{p,\gvr}^\ast = \sup_{y\in\R^d}\frac{1}{(1 + \|y\|)^\gvrp}\Bigl(\int_{\R^d}\bigl| f_j(x + y) - f_k(x + y)\bigr|^p d\nu(x)\Bigr)^{1/p}\to 0
\end{equation}
as $j,k$ approaches infinity. Thus, we have to show that there exists $f\in \bar{\LLl}^p_\gvr(\nu)$ such that $f_j\to f$ for $j\to \infty$ in $\bar{\LLl}^p_\gvr(\nu)$. Due to
\[
\|f - f_j\|_{p,\gvr}^\ast \le \|f - f_{j_i}\|_{p,\gvr}^\ast + \|f_{j_i}- f_j\|_{p,\gvr}^\ast
\]
it remains to verify this fact for a subsequence of $(f_j)_{j\in\N}$. 
We choose the subsequence $(f_{j_m})_{m\in\N}$ such that $(f_{j_m})$ converges a.s. and 
\begin{equation}\label{eq-lem-construction of limes}
 \sum_{m=1}^\infty \|f_{j_{m+1}} - f_{j_m}\|_{p,\gvr}^\ast < \infty.
\end{equation}
and denote $(f_{j_m})$ by $(f_m)$ again. Since 
 \[
 \frac{1}{(1 + \|0\|)^\gvrp}\Bigl(\int_{\R^d} |g( x + 0)|^p d\nu(x)\Bigr)^{1/p} \le \sup_{y}\frac{1}{(1 + \|y\|)^\gvrp}\Bigl(\int_{\R^d}\bigl| g(x + y)\bigr|^p d\nu(x)\Bigr)^{1/p}
 \]
for all $g\in \bar{\LLl}^p_\varrho(\nu)$, we have $\|f_m\|_p\le \|f_m\|_{p,\gvr}^\ast$. Due to \eqref{eq-lem-cauchy sequence definition} the sequence $(f_m)_{m\in\N}$ is convergent in $\LLl^p(\nu)$ and there exists a limit $f\in\LLl^p(\nu)$. 
By the Lemma of \emph{Fatou} we obtain for $y\in\R^d$:
\begin{align*}
\lefteqn{\frac{1}{(1 + \|y\|)^\gvr}\int  |f(x+y) - f_m(x+y)|^pd\nu(x)}\quad \\
& \le  \liminf_{l\to\infty} \frac{1}{(1 + \|y\|)^\gvr}\int|f_l(x + y) - f_m(x + y)|^p d\nu(x)\\
& \le  \liminf_{l\to \infty}\Bigl(\|f_l - f_m\|_{p,\gvr}^\ast\Bigr)^p\\
& \le \lim_{l\to \infty}\Bigl(\sum_{k=m}^{l-1}\|f_{k+1} - f_k\|_{p,\gvr}^\ast\Bigr)^p
 = \Bigl(\sum_{k=m}^{\infty}\|f_{k+1} - f_k\|_{p,\gvr}^\ast\Bigr)^p <\infty.
\end{align*}
Thus, we have 
\[
\sup_{y\in\R^d}\frac{1}{(1 + \|y\|)^\gvrp} \Bigl(\int|f(x+y) - f_m(x+y)|^pd\nu(x)\Bigr)^{1/p} \le\Bigl(\sum_{k=m}^{\infty}\|f_{k+1} - f_k\|_{p,\varrho}^\ast\Bigr).
\]
By \eqref{eq-lem-construction of limes} the last term tends to zero as $m$ approaches infinity. This implies the statement of the lemma.
\end{proof}

For time-inhomogeneous translation invariant Markov process $(X_t)_{t\ge 0}$ on $\R^d$ with transition function $(P_{s,t})_{s\le t}$ satisfying a certain integrability condition the ES-property holds for $\bar{\LLl}^p_{\text{c},\gvr}(\nu) = \bar{\LLl}^p_\varrho(\nu) \cap C(\R^d)$, the set of continuous functions in $\bar{\LLl}^p_\varrho(\nu)$, with respect to $\|\cdot\|_{p,\gvr}^\ast$-norm as well.

\begin{propos}\textbf{\emph{(ES-property for time-inhomogeneous translation invariant Markov processes)}}%
\label{propos-parade Beispiel modified L p space is well suited for trans inv MP}
Assume that $\|z\|^\gvr$ is uniformly integrable w.r.t. $P_{s,t}$ for some $\gvr > 0$ with $s\le t$, i.e.,
\[
\sup_{s\le t} \int_{\|z\|\ge K} \|z\|^s P_{s,t} (dz) \stackrellow{\longrightarrow}{K\to\infty} 0.
\]
Then, the family of transition operators $T=(T_{s,t})_{s\le t}$ defined for $f\in\bar{\LLl}^p_{\text{c},\varrho}(\nu)$ by \eqref{eq-propos-expactation operator for tranlation inv MP}, is a strongly continuous ES on $\bar{\LLl}^p_{\text{c},\varrho}(\nu)$ for $1\le p < \infty$. 
If additionally, $\nu$ is absolute continuous with respect to the Lebesgue measure $\llambda$, then $T$ is a strongly continuous $ES$ on $\bar{\LLl}^p_{\varrho}(\nu)$.
\end{propos}

\begin{proof}
 The proof for the strong continuity is analogous to the proof of Proposition \ref{propos-parade Beispiel for evo systems}. Hereto, assume $\nu$ to be absolute continuous with respect to $\llambda$. For $f\in \bar{\LLl}^p_{\varrho}(\nu)$ note that the function
 \[
 h^\ast(z) = \sup_y \frac{1}{(1 + \|y\|)^\gvrp} 
\Bigl(\int |f(x + y + z) - f(x + y)|^p d\nu(x) \Bigr)^{1/p}
 \]
is bounded by $c(1 + \|z\|^\gvr)$ and continuous with $h^\ast(0) = 0$ where $c=c_f$ is constant in $\R_+$. In the case where $\nu$ is not absolute continuous w.r.t. $\llambda$, then simply restrict the function class to $f\in \bar{\LLl}^p_{\text{c},\varrho}(\nu)$. 
Consequently, $h^\ast$ becomes continuous again. For the boundedness let $f\in \bar{\LLl}^p_{\varrho}(\nu)$ and observe that the expression 
 \[
 \Bigl(\frac{1}{(1 + \|y\|)^\gvr}(1 + \| y + z\|)^\gvr\Bigr)^{1/p}
 \] 
assumes its maximum in $y=0$. Thus, 
 \begin{align*}
 h^\ast(z)& \le  c\Bigl(\sup_y \frac{1}{(1 + \|y\|)^\gvrp} \Bigl(\int |f(x + y + z)|^p d\nu(x)\Bigr)^{1/p} + \|f\|_{p,\gvr}^\ast \Bigr) \\
& \le  c\Bigl(\sup_y \frac{(1 + \|y + z\|)^\gvrp}{(1 + \|y\|)^\gvrp} \\
& \quad {} \cdot \sup_{y+z} \frac{1}{(1 + \|y + z\|)^\gvrp} \Bigl(\int |f(x + y + z)|^p d\nu(x)\Bigr)^{1/p} + \|f\|_{p,\gvr}^\ast \Bigr) \\ 
& \le  c\|f\|_{p,\gvr}^\ast ( 1 + \|z\|^\gvr)^{1/p}
 \end{align*}
for a sufficiently large $c>0$. Then we have
 \begin{align*}
 {{\|T_{s,t}f - f\|^{\ast}}^p_{p,\gvr}} & = \sup_y \frac{1}{(1 + \|y\|)^{\gvr}} \Bigl(\int |T_{s,t}f(x + y) - f(x + y) |^p d\nu(x)\Bigr) \\*
 & \le \int P_{s,t}(dz) {h^\ast}^p(z).
 \end{align*}
Note that for each $\varepsilon>0$ we can find a $\delta >0$ such that $\int_{\{|z| < \delta\}} P_{s,t}(dz){h^\ast}^p(z)< \varepsilon$ due to the continuity of $h^\ast$. 
Since ${h^*}^p$ is bounded by the uniformly integrable function $z\mapsto 1 + \|z\|^\gvr$ w.r.t. $P_{s,t}$ separating the integral and letting $s,t\to u$, $u\in\R_+$, $s<t$ we obtain for $M\in \R_+$ that
 \begin{align*}
 \lim_{\substack{s\to u, t\to u,\\s<t}} \int_{\R^d} P_{s,t}(dz) {h^\ast}^p(z)
& \le  \varepsilon + \lim_{s\to u,\, t\to u}\int_{\{|z| > \delta\}\cap [-M,M]} P_{s,t}(dz) {h^\ast}^p(z)\\*
 & \quad{}+ \lim_{s\to u,\, t\to u}\int_{\{|z| > \delta\}\cap [-M,M]^{\textsl{c}}} P_{s,t}(dz) {h^\ast}^p(z) \\*
& \le  \varepsilon + \int_{\{|z| > \delta\}\cap [-M,M]} P_{u,u}(dz) {h^\ast}^p(z)\\
 & \quad{} + \lim_{s\to u,\, t\to u}\int_{\{|z| > \delta\}\cap [-M,M]^{\textsl{c}}} P_{s,t}(dz) {h^\ast}^p(z) \\
& \le  \varepsilon + \int_{\{|z| > \delta\}\cap [-M,M]} \delta_0(dz) {h^\ast}^p(z) + \varepsilon 
 = 2 \varepsilon. 
 \end{align*}
Moreover, recall that $\int_{\R^d} \|z\|^\gvr P_{s,t}(dz)<\infty$, thus
 \begin{align*}
 {\|T_{s,t}f\|_{p,\gvr}^\ast} 
& =  \sup_y \frac{1}{(1 + \|y\|)^\gvrp} \Bigl(\int_{\R^d} |T_{s,t} f( x + y)|^p d\nu(x)\Bigr)^{1/p} \\
 & =  \sup_y \frac{1}{(1 + \|y\|)^\gvrp} \Bigl(\int_{\R^d} 
\Bigl| \int_{\R^d} f( x + y +z) P_{s,t}(dz)\Bigr|^p d\nu(x)\Bigr)^{1/p} \\
 & \le  \sup_y \frac{1}{(1 + \|y\|)^\gvrp} \Bigl(\int_{\R^d} P_{s,t}(dz) \Bigl(\int_{\R^d} |f( x + y +z)|^p d\nu(x)\Bigr) \Bigr)^{1/p}.
\end{align*}
Consequently, we arrive at
\begin{align*}
{\|T_{s,t}f\|_{p,\gvr}^\ast} 
& \le 
 \|f\|_{p,\gvr}^\ast \sup_y \Bigl(\int_{\R^d} P_{s,t}(dz) \Bigl( \frac{1}{(1 + \|y\|)^\gvr}(1 + \| y + z\|)^\gvr\Bigr)\Bigr)^{1/p} \\*
	& \le  \|f\|_{p,\gvr}^\ast \Bigl(\int_{\R^d} P_{s,t}(dz) (1 + \|z\|)^\gvr\Bigr)^{1/p} \\*
	& \le  c\cdot\|f\|_{p,\gvr}^\ast \Bigl(\int_{\R^d} \|z\|^\gvr \,P_{s,t}(dz)\Bigr)^{1/p} 
			 \le c'\|f\|_{p,\gvr}^\ast.
 \end{align*}
Hence, $\|{T_{s,t}\arrowvert}_{\bar{\LLl}^p_{\varrho}(\nu)}\| \le c'$ and $T_{s,t}f \in \bar{\LLl}^p_{\varrho}(\nu)$ for all $0\le s\le t$, i.e. $T$ is a strongly continuous ES on $\bar{\LLl}^p_{\varrho}(\nu)$.
\end{proof}

\begin{Remark} \textbf{(}\boldmath$b-$\textbf{bounded functions)} \label{rem-convex generator is in modified L p space} \\
The norm modified $\LLl^p$-space, $\bar{\LLl}^p_\gvr(\nu)$ allows us to deal with orderings generated by unbounded function classes. Let $b:\R^d\to[1,\infty)$ be a weight function and define 
\[
B_{\text{b}}:=\bigl\{f:\R^d\to\R \mid \exists\, c\in\R: |f(x)|\le c\cdot b(x)\bigr\}
\] 
the class of $b$-bounded functions. In particular, to deal with the convex ordering we choose $b(x) = 1 + \|x\|$ for $\|\cdot\|$ any norm on $\R^d$. Then the class of $b$-bounded convex functions is a generator of the convex ordering $\sord{cx}$ and is a subset of $\bar{\LLl}^p_\gvr(\nu),\,p\le \gvr$ for any $\nu$ which integrates $(1 + \|x\|)^p$. 
\end{Remark}

Thus in the case of $\bar{\LLl}^p_{\varrho}(\nu)$ the $ES$-property is implied by the translation invariance property and a integrability condition of the transition function $(P_{s,t})_{s\le t}$. Under an additional assumption on the transition kernel we extend in the following the ES-property to the case of not necessarily translation invariant time-inhomogeneous Markov processes for unbounded function classes. 
More specifically we consider the class $\bar \LLl^2_2(\nu)$ and impose the following assumption $(K)$, which is a strengthening of the notion of the Hilbert--Schmidt operator:

We assume that the transition kernel $P_{s,t} (s,dz)=k_{s,t} (x,x+z)\nu(dz)$ corresponding to a time-inhomo-geneous Markov process has a density w.r.t. $\nu$ fulfills
\begin{flalign} 
\text{\upshape(K)}&& &K_{s,t}^* := \sup_{y\in\R^d} \frac{1}{1+\|y\|}\bigg( \int_{\R^d} |k_{s,t} (y,y+z)|^2 d\nu(z)  \bigg)^{1/2}<\infty & \nonumber
\end{flalign}
and satisfies the continuity assumption
\begin{flalign} 
\text{\upshape(C)}&& &\text{For } 0<s<t \text{ holds } \lim_{s'\to s,t'\to t} K_{s',t'}^* = K_{s,t}^*. &\nonumber
\end{flalign}
For the continuity in $(s,s)$ we make the following local domination assumption.
\begin{flalign} 
\text{\upshape(D)} \kern-2ex&& &\text{For any } s\ge 0 \text{ there exists } \delta=\delta(s)>0, C>0, \text{ and a translation in-} & \nonumber \\[-.5ex]
&&& \text{variant transition function } (Q_{s,t}) \text{ such that } \|z\|^2 \text{ is uniformly integrable}\kern-4ex &  \nonumber \\[-.5ex]
&&& \text{w.r.t. } (Q_{s,t}) \text{ and } P_{s,t} f\le C Q_{s,t} f, \, \forall f\in\overline{\mathcal L}\vphantom{\mathcal L}_2^2(\nu), f\ge 0, t-s\le\vartheta. & \nonumber
\end{flalign}
Note that conditions (K), (C) are fulfilled under boundedness and continuity assumptions on $k$. The domination condition (D) ensuring continuity of $T_{s,t}$ in $(s,s)$ might be verified in several applications but could also be replaced by further ad hoc assumptions.

\begin{propos}\textbf{\emph{(ES-property for time-inhomogeneous Markov processes on \boldmath$\bar{\LLl}^2_{2}(\nu)$)}\unboldmath}%
\label{propos-trans func of general MP is evo system on L^2(nu)} \\
Let $\nu$ be a finite measure with finite second moments on $\R^d$ and let $(X_t)_{t\ge 0}$ be a time inhomogeneous Markov process with transition kernel $(P_{s,t})_{s\le t}$ which satisfies the assumptions (K), (C) and (D). Then the corresponding family of transition operators $(T_{s,t})_{s\le t}$ is a strongly continuous ES on $\overline{\mathcal L}\vphantom{\mathcal L}_2^2(\nu)$.
\end{propos}

\begin{proof}
For $f\in \bar{\LLl}_2^2(\nu)$ and $x,y\in\R^d$ holds that
\begin{align*}
\lefteqn{|T_{s,t} f(x+y)| }\quad \\
& \le \bigg| \int k_{s,t} (x+y,x+y+z) f(x+y+z) \nu(dz)\bigg| \\
&\le \bigg( \int| k_{s,t} (x+y, x+y+z)|^2 d\nu(z)\bigg)^{1/2} 
\bigg( \int f^2(x+y+z) d\nu(z)\bigg)^{1/2}.
\end{align*}
This implies that 
\begin{align*}
 A_{s,t} &:= \sup_y \frac{1}{1+\|y\|} \bigg( \int (T_{s,t} f(x+y))^2 d\nu(x)\bigg)^{1/2} \\*
&\le \sup_y \frac{1}{1+\|y\|} \biggl(\int\biggl(\int k_{s,t}^2 (x+y,x+y+z) d\nu(z) \\*
& \hspace*{22ex} \int f^2(x+y+z) d\nu(z) \biggr) d\nu(x)\biggr)^{1/2}.
\end{align*}
We also obtain
\begin{align*}
\lefteqn{\frac{1}{1+\|y\|} \int f^2(x+y+z) d\nu(z) }\quad \\
& = \frac{1+\|x+y\|}{1+\|y\|} \frac{1}{1+\|x+y\|} \int f^2(x+y+z)d\nu(z) \\
& \le (1+\|x\|) \|f\|_{2,2}^*
\end{align*}
and
\begin{align*}
\sup_y \frac{1}{1+\|y\|} \int k_{s,t}^2 (x+y,x+y+z) d\nu(z) &\le (1+ \|x\|) K_{s,t}^*.
\end{align*}
As consequence we get
\[
A_{s,t} \le \bigg(\int (1+\|x\|)^2 d\nu(x)\bigg)^{1/2} \|f\|_{2,2}^* K_{s,t}^*.
\]
Thus by assumption (K) $T_{s,t} f\in\bar\LLl_2^2(\nu)$ for $f\in\bar\LLl_2^2(\nu)$ and by the continuity assumption (C) $T_{s,t}$ is strongly continuous in $(s,t)$ for $s<t$. By the domination assumption (D) we obtain from Proposition \ref{propos-parade Beispiel modified L p space is well suited for trans inv MP} that $T_{s,t}$ is also continuous in $(s,s)$, $0\le s$.
\end{proof}

Several modifications of the domination assumption ensuring continuity in $(s,s)$ could be given. \vspace{2mm}

\noindent\textbf{Example: PII and their infinitesimal generators}
\label{rem-infini gen of time inhomo MP} \\
Let $L = (L_t)_{t\ge 0}$ be a PII with continuity property \eqref{eq:continuity condition for trans kernel} and with characteristic function given by
 \begin{equation}
\label{eq-ex-characteristic function of PII}
 Ee^{i\langle\xi,L_t\rangle} = e^{\int_0^t\theta_s(i\xi)ds},
 \end{equation}
where for $s\ge 0$ the cumulant function $\theta = (\theta_s)_{s\ge0}$ equals
 \begin{equation}\label{eq-ex-cumulant fucntion of PII}
 \theta_s(i\xi) = -\frac{1}{2}\langle\xi,\sigma_s\xi\rangle + i\langle\xi,b_s\rangle + \int\left(e^{i\langle\xi,y\rangle}-1-i\langle\xi,\chi_{\texttt C}(y)\rangle\right)F_s(dy)
 \end{equation}
for a \emph{cut-off} function $\chi_{\texttt C}$, that is a bounded, measurable real function on $\R^d$ with compact support and which equals the identity in a neighbourhood of zero. 
Here for each $s>0$ the \emph{covariance matrix} $\sigma_s$ is a symmetric, positive semi-definite $d\times d$ matrix, the \emph{drift} $b_s$ is in $\R^d$ and $F_s$ is a L\'{e}vy-measure, i.e.\ a Borel-measure on $\R^d$ which integrates $(1\wedge |x|^2)$ with $F(\{0\})=0$. 

For a triplet $(b_s,\sigma_s,F_s)$ defined as above we consider the operator $G_sf(x) = G^D_sf(x) + G^J_sf(x),\,s\in\R_+$ with 
 \begin{equation}\label{eq-jump term of infini gen}
 \begin{split}
 G^D_sf(x) & := \frac{1}{2}\sum_{j,k=1}^d\sigma_s^{j,k}\frac{\partial^2f}{\partial x_j\partial x_k}(x) + \sum_{j=1}^d b_s^j\frac{\partial f}{\partial x_j}(x), \text{ and}\\ 
 G^J_sf(x) & := \int_{\R^d}\Bigl(f(x+y) - f(x) -\sum_{j = 1}^d\frac{\partial f}{\partial x_j}(x)(\chi_{\texttt C}(y)^j)\Bigr)F_s(dy) \\
 \end{split}
 \end{equation}
on $\bar{\WW}^2(\nu) = \Bigl\{f\in C^2(\R^d) \cap \bar{\LLl}^2_2(\nu) | \frac{\partial f}{\partial x_i}, \frac{\partial^2 f}{\partial x_i \partial x_j} \in \bar{\LLl}^2_2(\nu)\, \forall\, 1\le i,j\le d \Bigr\} \subset \bar{\LLl}^2_2(\nu)$. 

For upcoming results it is crucial to know that $G_sf$ belongs to $\bar{\LLl}^2_2(\nu)$ and whether such an operator can be linked to a strongly continuous $ES$ on $\bar{\LLl}^2_2(\nu)$ corresponding to a time-inhomogeneous Markov process. 
Obviously, for $f\in \bar{\WW}^2(\nu)$ the expression $G^D_sf \in \bar{\LLl}^2_2(\nu)$ for every $s\in \R_+$. 
Moreover, under some mild regularity conditions it can be shown that ${G_s^J}:\bar{\WW}^2(\nu) \to \bar{\LLl}^2_2(\nu)$, where $\bar{\WW}^2(\nu)$ is a subspace of $\bar{\LLl}^2_2(\nu)$ with norm $\|\cdot\|_{2,2}^\ast$.

\begin{lemma}\label{lem-diff operator is OKAY}
 For the operator $G^D_sf$ introduced in equation \eqref{eq-jump term of infini gen} it holds that ${G_s^D|}_{\bar{\WW}^2(\nu)}:\bar{\WW}^2(\nu) \to \bar{\LLl}^2_2(\nu)$. If additionally 
 \[
 \int_0^t \int_{\{|y|\ge 1\}}|y|^2 F_s(dy) ds < \infty,
 \]
then the same holds true for the operator ${G_s|}_{\bar{\WW}^2(\nu)}$.
\end{lemma}

\begin{proof}
 It remains to show that for $f\in \bar{\WW}^2(\nu)$ the jump part $G^J_sf$ in \eqref{eq-jump term of infini gen} belongs to $\bar{\LLl}^2_2(\nu)$. We only cover the one-dimensional case. 
The multivariate version is similar. Let $f\in \bar{\WW}^2(\nu),\,s\in\R_+$, and choose the cut-off function $\chi_{\texttt C}(y) = y\1_{\{|y| < 1\}}$. Using the Taylor expansion for $x'\in \R$ we arrive at 
 \begin{align*}
 \Bigl|\int_{\R} f(x' +y)- & f(x')-f'(x')y\1_{\{|y|<1\}}F_s(dy)\Bigr|\\
	& \le  \Bigl|\int_{\R} \bigl(f(x'+y)-f(x')-f'(x')y\bigr)\1_{\{|y|<1\}}F_s(dy)\Bigr| \\
			 & \quad {}+ \int_{\{|y|\ge 1\}}|f(x'+y)-f(x')|F_s(dy)\\
	& \le  \underbrace{\int_{\{|y|<1\}} \frac{1}{2}y^2 |f''(\xi)| F_s(dy)}_{=:I_1} 
			 + \underbrace{\int_{\{|y|\ge 1\}}|f(x'+y)|F_s(dy)}_{=:I_2} \\
	& \quad{}+ \underbrace{|f(x')| \int_{\{|y|\ge 1\}} (1 \wedge |y|^2) F_s(dy)}_{=:I_3},
 \end{align*}
where $\xi$ is a suitable intermediate point between $x'$ and $x'+y$ depending on $y$. Thus, the term $|f''(\xi)|$ is bounded on $\{|y|<1\}$ since $f\in C^2(\R^d)$. For $x,z\in\R,$ we choose $x'= x + z$, then, squaring and integrating each term successively w.r.t. $d\nu(x)$. For $I_3$ we then have
 \begin{align*}
 \Bigl(\int_{\{|y|\ge 1\}} (1 \wedge |y|^2) F_s(dy)\Bigr)^2\cdot & \int |f(x + z)|^2 d\nu(x) \\*
		& \le   (1 + |z|)\cdot c_3 \sup_{z'} \frac{1}{1 + |z'|} \Bigl(\int |f(x + z)|^2 d\nu(x)\Bigr) \\
		& \le  (1 + |z|)\cdot c_3 (\|f\|_{2,2}^\ast )^2
 \end{align*}
for a suitable non-negative constant $c_3$. Since
 \[
 \int_0^t \int_{\{|y|\ge 1\}}|y|^2 F_s(dy) ds < \infty,
 \]
for the middle term $I_2$ we observe that
 \begin{align*}
 \lefteqn{\int \int_{\{|y|\ge 1\}} |f(x+z+y)|^2F_s(dy) d\nu(x) }\quad \\ 
& =  \int_{\{|y + z|\ge 1\}} (1 + |y+ z|) \Bigl(\frac{1}{1 + |y+ z|} \int |f(x + y+ z)|^2 d\nu(x) \Bigr)F_s(dy) \\
	 & \le  \int_{\{|y|\ge 1\}} (1 + |z + y|)\cdot \sup_{z'} \frac{1}{1 + |z' + y|} \Bigl(\int |f(x+z + y)|^2 d\nu(x) \Bigr)F_s(dy) \\
	 & \le  (\|f\|_{2,2}^\ast )^2 \Bigl( \int_{\{|y|\ge 1\}} (1 + |z|)(1 + |y|) F_s(dy) \Bigr)\\
	 & \le  (1 + |z|)(\|f\|_{2,2}^\ast )^2 \Bigl( \int_{\{|y|\ge 1\}} (1 \wedge |y|^2) F_s(dy) + \int_{\{|y|\ge 1\}} |y|^2 F_s(dy)\Bigr) \\
	 & \le  (1 + |z|)\cdot c_2 (\|f\|_{2,2}^\ast )^2,
 \end{align*}
where in the first equality Fubini's theorem is used, and again $c_2$ is a suitable non-negative constant. 
Using that $f''\in\overline{\mathcal L}_2^2(\nu)$ we obtain from Taylor expansion that the first term $I_1$ is bounded by a constant $c_1$. In summary this yields
 \begin{align*}
 \int_{\R} & |G^J_sf(x + z)|^2 d\nu(x)\\
	& =  \int_\R \Bigl|\int_{\R} f(x+z+y)-f(x+z)-f'(x+z)y\1_{\{|y|<1\}} F_s(dy)\Bigr|^2d\nu(x) \\
	& \le  (1 + |z|)\cdot(c_1 + c_2 (\|f\|_{2,2}^\ast)^2 +c_3(\|f\|_{2,2}^\ast)^2),
 \end{align*}
which implies
 \[
 \sup_z \frac{1}{1 + |z|} \Bigl(\int_{\R} |G^J_sf(x + z)|^2 d\nu(x)\Bigr) \le (c_1 + c_2 (\|f\|_{2,2}^\ast)^2 +c_3(\|f\|_{2,2}^\ast)^2) <\infty.
 \]
Consequently, we proved $\|G^J_sf\|_{2,2}^\ast < \infty $, hence, the statement holds true.
\end{proof}

\begin{propos} \textbf{\emph{(Infinitesimal generator for PII on \!\boldmath$\bar{\LLl}^2_2(\nu)$\unboldmath)}}
\label{prop-Diff operator is infini gen} \\
Let $(L_t)_{t\ge 0}$ be a PII such that $E|L_t|^2<\infty$ for all $t\in\R_+$. Denote by $T=(T_{s,t})_{s\le t}$ the corresponding transition operator
on $\bar{\LLl}^2_2(\nu)$, and its infinitesimal generator $A_sf$ defined via equation \eqref{eq-infi gen} for $f\in \D(A_s)$. Then, $A_sf = G_sf$ on $\bar{\WW}^2(\nu)$. In particular $\bar{\WW}^2(\nu)\subset \D(A_s)$ for each $s\in\R_+$.
\end{propos}

\begin{proof}
 From Proposition \ref{propos-parade Beispiel modified L p space is well suited for trans inv MP} we know that ${T_{s,t}|}_{\bar{\LLl}^2_{c,2}(\nu)}$ is a strongly continuous ES on the Banach space $\bar{\LLl}^2_2(\nu)$. Hence, the limit
 $A_sf=\lim_{h\to 0}\frac{1}{h}(T_{s,s + h}f - f) $ can be understood in the strong sense on $\D(A_s)$. Thus, $A_s:\D(A_s) \to \bar{\LLl}^2_2(\nu)$. 
Now, let $f\in \bar{\WW}^2(\nu)$ and consider a smooth approximation sequence\footnote{Utilize standard approximation by a mollifier $\varphi$ which is smooth up to the boundary.} $f_N\in C^2_c(\R^d),N\in \N$ such that $\partial f_N,\partial^2 f_N \in C^2_\text{b}(\R^d)$ with $\|f_N - f\|_{2,2}^\ast \to 0$. Recall that $\nu$ integrates $(1 + \|x\|^2)$, thus $f_N\in \bar{\LLl}^2_2(\nu)$ as well as $\partial f_N,\partial^2 f_N \in \bar{\LLl}^2_2(\nu)$, or equivalently, $f_N\in\bar{\WW}^2(\nu)$.
From standard literature\footnote{In \cite{Dynkin-1965} the infinitesimal generator is given for $g\in C_c^2(\R^d)$ for time-inhomogeneous L\'{e}vy processes and in \cite{Jacob-2001} for $g\in C^2_{\text b}(\R^d)$ for Markov processes.}, e.g. \cite{Dynkin-1965} and \cite{Jacob-2001}, we know that $A_sg = G_sg$ for all $g\in C^2_c(\R^d)$, where $G_s$ is the operator as defined in equation \eqref{eq-jump term of infini gen}. 
Consequently, we obtain $A_sf_N = G_sf_N$. The moment condition on $(L_t)_{t\geq 0}$ yields
\[
 \int_0^t \int_{\{|y|\ge 1\}}|y|^2 F_s(dy) ds < \infty,
\]
Lemma \ref{lem-diff operator is OKAY} then implies
 \[
 \|A_sf_N\|_{2,2}^\ast = \|G_sf_N\|_{2,2}^\ast \le \|G^D_sf_N\|_{2,2}^\ast + \|G^J_sf_N\|_{2,2}^\ast < \infty.
 \]
Since $G_sf_N$ and $G_sf \in \bar{\LLl}^2_2(\nu)$ for all $N\in\N$ it follows by majorization
 \[
 A_sf:=\lim_{N\to \infty} A_sf_N = \lim_{N\to \infty} G_sf_N = G_s(\lim_{N\to \infty} f_N) = G_sf .
 \]
In particular, it holds that $\|A_sf\|_{2,2}^\ast < \infty$ and, thus, $\bar{\WW}^2(\nu) \subset \D(A_s)$ for each $s\in \R_+$.
\end{proof}

Similar results as Lemma \ref{lem-diff operator is OKAY} and Proposition \ref{prop-Diff operator is infini gen} hold true under some mild regularity conditions, also for time-inhomogeneous pure diffusion processes.

\setcounter{equation}{0}
\section{Comparison of Markov processes}\label{sec-Comp of MP}
The main goal is to prove a general comparison result for time-inhomogeneous Markov process. Hereto, we need a representation result for strongly continuous ES on general Banach spaces. We start by recalling the weak evolution problem.

\subsection{Representation result for solutions of the weak evolution problem}\label{subsec-weak evo problem}
The \emph{weak evolution problem} is crucial for the comparison result in the next section. Let $\B$ be a Banach space. For a strongly continuous evolution system $T=(T_{s,t})_{s\le t}$ on $\B$ with corresponding family of infinitesimal generators $(A_s)_{s>0}$ define $F_t(s):=T_{s,t}f,\,s\le t$, $f\in\D(A_s)$. Then, by Lemma \ref{lem-basic lemma for evo system} for fixed $t$, $F_t$ is a solution of the homogeneous \emph{weak evolution equation}
\begin{align}
 \frac{d^+u(s)}{ds} & = -A_su(s)\text{ for }s< t,\label{eq-evolution equation for my solution op}\\
 u(t) & = f. \notag
\end{align}
As usual, $F_t$ is called the fundamental solution of \eqref{eq-evolution equation for my solution op}. In the autonomous case where $A_s$ is independent of $s$, equation \eqref{eq-evolution equation for my solution op} becomes the homogeneous Cauchy problem. 

Fundamental solutions and the Cauchy problem are studied in \cite{Friedman-1969} imposing conditions on the resolvent of the generator. In order to deal with inhomogeneous Markov processes we need to study solutions of the evolution problem in the following form which is related to \citet[Part 2, Chapter 3]{Friedman-1969}  (see also \citet[Chapter 5]{Pazy-1983}). This extension is the foundation of the comparison results in the next section.

For every $r\in[s,t]$, $\,s,t\in\R_+$ let $A_r:\D(A_r)\subset\B\to\B$ be a linear operator on $\B$ and let $G(r)$ be a $\B$-valued function on $[s,t]$. We consider for a $\B$-valued function $u$ on $[s,t]$, which is right differentiable on $(s,t),\,u(r)\in\D(A_r)$ for $s < r \le t$, the \emph{weak initial value problem} or \emph{weak evolution problem}
\begin{equation}\label{eq-initial value problem general}
\begin{split}
 \frac{d^+u(r)}{dr}  &=  -A_ru(r) + G(r)\text{ for }s< r \le t,\\
 u(t)  &=  f.
\end{split}
\end{equation}
A $\B$-valued function $u:[s,t]\to\B$ is a \emph{classical solution} of \eqref{eq-initial value problem general} if $u$ solves \eqref{eq-initial value problem general} and is continuous on $[s,t]$.

Note that the autonomous case of the weak evolution problem is a weakening of the inhomogeneous Cauchy problem. The basic representation result for solutions of the inhomogeneous Cauchy equation in a Banach space was used in \cite{Ruesch-Wolf-2010} to establish comparison results for homogeneous Markov processes. 
The following is an extension of this representation result to the weak evolution problem.

\begin{theorem}\textbf{\emph{(Representation result)}}\label{theo-evolution problem} \\
Let $T=(T_{s,t})_{s\le t}$ be a strongly continuous ES on $\B$ with corresponding family of infinitesimal generators $(A_s)_{s\ge0}$. Further, let $F_t,\,G:[0,t]\to\B$ for $t\in\R_+$ be functions such that
 \begin{enumerate}
 \item the map $r\mapsto T_{s,r}G(r)$ is right continuous.\label{enum-theo-right continuous}
 \item $\int_s^t T_{s,r}G(r)dr$ exists for all $s,t\in\R_+$ with $s\le t$.\label{enum-theo-integral exists}
 \item $F_t$ solves the weak evolution problem \eqref{eq-initial value problem general}, i.e. for $s,t\in\R_+$ such that $s\le t$, it holds that $r\mapsto F_t(r)$ is continuous on $[s,t]$, right differentiable on $(s,t)$, $F_t(r)\in\D(A_r)$ on $(s,t]$ and
 \begin{equation}
 \label{eq-theo-backward equation for F_t}\frac{d^+F_t(s)}{ds}= -A_sF_t(s) + G(s)\text{ for } s\le t.
 \end{equation}
 \end{enumerate}
Then,
\begin{equation}\label{eq-theo-solution to the evolution problem}
 F_t(s) = T_{s,t}F_t(t) - \int_s^tT_{s,r}G(r)dr.
\end{equation}
\end{theorem}

\begin{proof}
Let $s,t\in\R_+$ with $s< t$. The right derivatives $\frac{d^+}{dr}T_{s,r}\text{ and }\frac{d^+}{dr}F_t(r)$ exist due to the assumption on $T$ and $F_t$. Using \eqref{eq-ex-forward equation} and \eqref{eq-theo-backward equation for F_t} we obtain
\begin{align}
 \frac{d^+}{dr}\Bigl(T_{s,r}F_t(r)\Bigr)
 & = \lim_{h\downarrow 0}T_{s,r+h}\frac{1}{h}\Bigl(F_t(r+h) - F_t(r)\Bigr) + \lim_{h\downarrow 0}\frac{1}{h}\Bigl(T_{s,r+h} - T_{s,r}\Bigr)F_t(r)\notag\\
 & =  T_{s,r} \Bigl(\frac{d^+}{dr}F_t(r)\Bigr) + \Bigl(\frac{d^+}{dr}T_{s,r}\Bigr)F_t(r)\notag\\
 & =  T_{s,r}\Bigl(-A_rF_t(r) + G(r)\Bigr) + T_{s,r}A_rF_t(r)\notag\\
 & =  T_{s,r}G(r)\label{eq-theoproof-differentiating of T before integrating}.
\end{align}
Thus $\frac{d^+}{dr}\bigl(T_{s,r}F_t(r)\bigr)$ is integrable on $[s,t]$ due to assumption \ref{enum-theo-integral exists}. Since it is right continuous we can integrate \eqref{eq-theoproof-differentiating of T before integrating} and obtain
\[
\int_s^t T_{s,r}G(r) dr = \int_s^t\frac{d^+}{dr}\bigl(T_{s,r}F_t(r)\bigr) dr = T_{s,t}F_t(t) - T_{s,s}F_t(s).
\]
Thus, we obtain the representation of the solution of the weak evolution problem
\[
F_t(s) = T_{s,t}F_t(t) - \int_s^t T_{s,r}G(r) dr.
\]
~\\[-8.5ex]
\end{proof}
\message{^^J---^^J---^^J---^^J--- Take care for this end-of-proof, needed for positioning the qed at the correct position within the formula line. ---^^J---^^J---^^J}

\begin{Remark}
\begin{description}
 \item[(a)] Using the theory of pseudo-differential operators, \cite{Boettcher-2008} developed an integral representation result on certain subspaces of $C_0$ for solutions of an evolution problem related to \eqref{eq-initial value problem general}. See also our Section \ref{sec-general remark}.
 \item[(b)] For a special class of infinitesimal generators of strongly continuous semigroups \citet[Chapter 5, Theorem 4.2]{Pazy-1983}  derived a similar representation formula for evolution systems on Banach spaces.
\end{description}
\end{Remark}

\subsection{General comparison result for Markov processes}\label{subsec-gen comp for MP}
We assume that $X$ and $Y$ are two time-inhomogeneous Markov processes with values in $(\E,\A)$ and let $S=(S_{s,t})_{s\le t}$ and $T=(T_{s,t})_{s\le t}$ denote their strongly continuous evolution systems on some Banach function space $\B$ on $\E$. Let by $A=(A_s)_{s\ge0}$ and $B=(B_s)_{s\ge 0}$ denote the corresponding families of infinitesimal generator of $S$ and $T$ respectively. 
Further, we assume that for each $s>0$
\begin{equation}\label{eq-F is in domains}
 \F\subset\D(A_s)\cap\D(B_s).
\end{equation}

\begin{theorem}[Conditional comparison result]\label{theo-comparison result}
 Assume that
 \begin{enumerate}
 \item $f\in\F$ implies $s\mapsto T_{s,t}f$ is right-differentiable for $0<s<t$,\label{enum-theo-T leaves invariant}
 \item $f\in\F$ implies $S_{s,t}f\in\F$ for $0\le s\le t$,\, \emph{(stochastic monotonicity of $S$)}\label{enum-theo-X is stoch mon}
 \item \raisebox{-.3ex}{%
 \parbox{\linewidth}{%
 \begin{equation}\label{eq-theo-fsord of infini gen}
 A_sf\le B_sf ~[P^{X_0}]\text{ for all }f\in\F\text{ and }s\le t.
 \end{equation}}}
 \end{enumerate}
 Then \begin{equation} S_{s,t}f\le T_{s,t}f ~[P^{X_0}] \text{ for all }f\in\F,\,s\le t.\end{equation}
\end{theorem}

\begin{proof}
 Define for $f\in\F$ and $t\in\R_+$ the function $F_t:[0,t]\to\B$ by $F_t(s):= T_{s,t}f - S_{s,t}f$. Then due to Lemma \ref{lem-basic lemma for evo system} $F_t$ satisfies the differential equation
\[
\frac{d^+F_t(s)}{ds} = -B_sT_{s,t}f + A_sS_{s,t}f,\,F_t(t) = 0,\text{ on $[0,t]$}.
\]
As consequence we obtain
\begin{align*}
 \frac{d^+F_t(s)}{ds} & =  -B_s(T_{s,t}f - S_{s,t}f)+ \Bigl(A_s - B_s\Bigr)S_{s,t}f= -B_sF_t(s) + G(s),
\end{align*}
where by \eqref{eq-F is in domains} and assumption \ref{enum-theo-X is stoch mon} $G(s) := \bigl(A_s - B_s\bigr)S_{s,t}f$ is well defined. Further it is non-positive due to assumption \eqref{eq-theo-fsord of infini gen}. Moreover we obtain that $\bigl(-B_sF_t(s)\bigr)$ is also well defined due to assumptions \ref{enum-theo-T leaves invariant}, Lemma 2.1 and \ref{enum-theo-X is stoch mon}. 
The continuity of $S$ and $T$ transfers to $F_t(s)$. Furthermore, the strong continuity of $T_{s,t}$ and $S_{s,t}$ transfers to the map $r \mapsto T_{s,r}G(r)$, hence, it is right continuous, too. Thus, $F_t(s)$ solves the weak evolution problem. 
Since $-G(r)\ge0$ for $s\le r \le t$ it follows that $-\int_s^tT_{s,r}G(r)dr$ exists and is finite for all $s,t\in\R_+$. 
Thus, the assumptions of Theorem \ref{theo-evolution problem} are satisfied and imply that the solution $F_t(s)$ has an integral representation of the form
 \[
 F_t(s) = T_{s,t}F_t(t) - \int_s^t T_{s,r}G(r) dr
	 = \int_s^t T_{s,r}\bigl(-G(r)\bigr)dr\ge 0,
 \]
as $F_t(t) = 0$ and $-G(r)\ge 0$ for all $s\le t$.
\end{proof}
\begin{Remark}
\begin{description}
 \item[(a)] 
The comparison result for homogeneous Markov processes in \citet[Theorem 3.1]{Ruesch-Wolf-2010} is implied by Theorem \ref{theo-comparison result} since in the case of homogeneous Markov processes the corresponding semigroups fulfil condition \ref{enum-theo-T leaves invariant} (see for example Theorem 2.4 in \cite{Pazy-1983}).\label{rem-evolution problem implies Cauchy problem}
 \item[(b)] 
Let $\F$ be any function class in \eqref{eq-fst def}--\eqref{eq-mixes def}. If the Markov process $X$ on $\R^d$ has a translation invariant transition function $(P_{s,t})_{s\le t}$, then $X$ is stochastically monotone w.r.t.\ to $\F$(see Lemma 2.8 in \cite{Bergen-Ruesch-2007}) and thus condition \ref{enum-theo-X is stoch mon} in Theorem \ref{theo-comparison result} holds. \label{rem-translation invariant trans function}
\item[(c)] 
\textbf{\emph{(Smooth operator)}} Let $A_s$ be an operator as in Proposition \ref{prop-Diff operator is infini gen}, then, for transition operators of PII on $\bar{\LLl}^2_{\text c,2}(\nu)$ resp. $\bar{\LLl}^2_{2}(\nu)$ the condition \ref{enum-theo-T leaves invariant} reduces to show that $T$ is a \emph{smooth operator}. 
For instance, let $f\in \F^0 = \F \cap \bar{\WW}^2(\nu)$ and let the assumptions of Proposition \ref{propos-parade Beispiel modified L p space is well suited for trans inv MP} hold true. Then, for a smooth operator it holds that $T_{s,t}f\in C^2(\R^d)$. Further, due to Proposition \ref{propos-parade Beispiel modified L p space is well suited for trans inv MP}, it holds that $T_{s,t}f \in \bar{\LLl}^2_{2}(\nu)$. 
Since, $\partial f,\, \partial^2 f \in \bar{\LLl}^2_{2}(\nu)$ we also have $ T_{s,t}(\partial f)$ and $T_{s,t}(\partial^2 f) \in \bar{\LLl}^2_{2}(\nu)$. Hence, $T_{s,t}f \in \bar{\WW}^2(\nu) \subset \D(A_s)$.
\label{rem-smooth operator}
\end{description}
\end{Remark}

\setcounter{equation}{0}
\section{Applications}\label{sec-application}

\subsection{L\'{e}vy driven diffusion processes on \boldmath$C_{0}(\R^d)$} \label{subsec-Apps and Exs-LDDP}

In the recent financial mathematics literature L\'{e}vy driven diffusion processes have found considerable attention as stochastic interest rate models or stochastic volatility models. For a classical survey see \cite{Heston-1993}. Recent results and developments in this field can be found in \cite{Poulsen-etal-2009} and the references therein.

Let $(L_t)_{t\ge 0}$ be an $\R^d$-valued L\'{e}vy process with local characteristics $(b,\sigma, F)$. We consider the L\'{e}vy driven diffusion processes defined as solutions of the following stochastic differential equation:
\begin{equation}\label{eq-levy driven processes as SDGL}
\begin{split}
 dX_t & =  \Phi(X_{t-},t)dL_t\\
 X_0& =  x,\,x\in\R^d,
\end{split}
\end{equation}
where $\Phi:\R^d\times \R_+\to\R^{d\times d}$ is in $C^{1,1}_{\text{b}}(\R^d\times \R_+)$.

In the present section we will make use of the (probabilistic) symbol. This opens a neat way to derive the structure of the generator of L\'evy driven diffusions and, more general, to the class of Feller evolution processes.

\begin{definition}
Let $X$ be a Markov process with right-continuous paths (a.s) and 
\[
T_{s,t} u(x):= E^{x,s} u(X_t)=E(u(X_t)|X_s=x)
\]
for $s\leq t$.
If $T_{s,t}$ is a strongly continuous ES on the Banach space $C_0$ of continuous functions vanishing at infinity, then $X$ is called a \emph{$C_0$-Feller evolution process}. The process is called \emph{rich}, if
\[
C_c^\infty\subseteq \D(A_s) \text{ for every } s\geq 0.
\]
\end{definition}

Let us mention that in some parts of the literature $C_b$ is used instead of $C_0$. This larger space is then endowed usually with the strict topology (cf. Van Casteren (2011), Section 2.1). Here, it is convenient to use $C_0$ as reference space. 
We will see in the proof of Lemma \ref{sde-symbol} below that in the setting \eqref{eq-levy driven processes as SDGL} we obtain a $C_0$-Feller evolution process for \emph{every} L\'evy process. If we used $C_b$ instead, we would either loose the Banach space structure (using the strict topology which is defined via a family of semi-norms) or we would loose strong continuity of the semigroup in most cases (using the sup-norm). 
The following concept of a probabilistic symbol has proved to be useful in the time-homogeneous case already. We introduce this notion for the time-inhomogeneous case in order to derive path properties of the process. In the present article we will use it to derive the structure of the generator. 

\begin{definition}
Let $X$ be an $\R^d$-valued Markov process, which is conservative and normal, that is $P^x(X_0=x)=1$, and having right-continuous paths (a.s.). Fix a starting point $x$ and a starting time $t\geq 0$ and define $\tau=\tau^x_R$ to be the first exit time from the ball of radius $R > 0$ after time $t$:
 \[
 \tau:=\tau^x_R:=\inf\big\{h\geq 0 : \nnorm{X_{t+h}^x-x} > R \big\}.
 \]
 The function $p:\R_+\times \R^d\times \R^d \rightarrow \C$ given by
 \begin{gather} \label{symbol}
 p(t, x,\xi):=- \lim_{h\downarrow 0}E^{x,t} \left(\frac{e^{i\langle X^\tau_{t+h}-x,\xi \rangle}-1}{h}\right)
 \end{gather}
 is called the (time-dependent) \emph{probabilistic symbol} of the process, if the limit exists for every $t\geq 0$ and $x,\xi\in\R^d$ independently of the choice of $R>0$.
\end{definition}

Let us recall that in the above definition $X^\tau_{t+h}=X_{\min\{t+h,\tau\}}$ denotes the process (which is started at time $t$ in $x$) stopped at time $\tau$. 

At first sight it might be a bit surprising that it is possible to demand that the limit does not depend on the choice of $R$ respectively $\tau$. Intuitively speaking, this is due to the fact that we are dealing with right-continuous processes. The symbol describes only the local dynamics of the process and this is not changed by using the stopping time. 

We need two auxiliary results in order to prove Theorem \ref{theorem:symbol}. The first one is the time-inhomogeneous version of Dynkin's formula. It follows from the well-known fact that $(M_h^{[x,t,u]})_{h\geq 0}$ given by
\[
 M_h^{[x,t,u]}:=u(X_{t+h}, t+h)-u(x,t)-\int_{t}^{t+h} (A_s + \partial_s) u(X_s,s) \,ds
\]
is a martingale for every $u\in \bigcap_{t\leq s \leq t+h} \D(A_s)\cap \D(\partial_s)$ with respect to every $P^{x,t}$, $x\in\R^d$, $t\geq 0$. To our knowledge the most general version of this result can be found in \cite[Theorem 2.11]{Casteren-2011}. There it is formulated for the $C_b$ case but it can be easily adapted to the $C_0$ case. For a function $u$ which depends on the time-inhomogeneous process, but not on time itself, we obtain:

\begin{lemma} \label{lem:dynkin}
Let $X$ be a Feller evolution process ($C_b$ or $C_0$) and $\tau$ a stopping time. Then we have
\begin{gather} \label{dynkin}
 E^{x,t} \Bigl[\int_t^{\tau \wedge (t+h)} A_s u(X_s) \,ds\Bigr] = E^{x,t} \bigl[u(X_{\tau \wedge (t+h)})\bigr]-u(x)
\end{gather}
for all $t>0$ and $u\in \bigcap_{t\leq s \leq t+h}\D(A_s)$.
\end{lemma}

In the previous lemma we have used that the class of martingales is stable under stopping. The following result has been established in \citet[Lemma 2.5]{Schilling-Schnurr-2010}.

\begin{lemma} \label{lem:folge}
Let $K\subset \R^d$ be a compact set. Let $\chi:\R^d \to \R$
be a smooth truncation function, i.e., $\chi\in C_c^\infty(\R^d)$
with
\[
 \1_{B_1(0)}(y) \leq \chi (y) \leq\1_{B_2(0)}(y)
\]
for $y\in\R^d$. Furthermore we define $\chi_n^x(y):=\chi((y-x)/n)$ and for $\xi\in\R^d$ $u_n^x(y):= \chi_n^x(y) e^{i \langle y,\xi\rangle}$. Then for sufficiently large $c>0$ we have for all $z\in K$
\[
 \left|u_n^x(z+y)-u_n^x(z)-\langle y,\nabla u_n^x(z)\rangle \1_{B_1(0)}(y)\right|
 \leq c \cdot \left(1 \wedge \abs{y}^2 \right).
\]
\end{lemma}

Now we are in a position to show that the probabilistic symbol (which is easy to calculate) and the functional analytic symbol (which we will use below) coincide.

\begin{theorem}\label{theorem:symbol}
 Let $X=(X_t)_{t\geq 0}$ be a conservative c\`adl\`ag Feller evolution process such that $C_c^\infty\subset D(A_s)$ $(s\geq 0)$. Then the generator $A_s|_{C_c^\infty}$ is a pseudo-differential operator (cf. Section \ref{sec-general remark} below) with symbol $-q(s, x,\xi)$, that is
\begin{equation}
 A_s f(x)=-(2\pi)^{-d/2}\int_{\R^d}e^{-i\langle x,\xi \rangle} q(s,x,\xi)\hat f(\xi)d\xi 
\end{equation}
where the symbol $q:\R_+\times \R^d \times \R^d \to \C$ has the following properties for fixed $s\geq 0$:
\begin{itemize} \itemsep -0.6ex
 \item $q(s,\cdot,\cdot)$ is locally bounded in $x,\xi$.
 \item $q(s,\cdot,\xi)$ is measurable for every $\xi\in\R^d$.
 \item $q(s,x,\cdot)$ is a continuous negative definite function in the sense of Schoenberg for every $x\in\R^d$.
\end{itemize}
 Let
 \begin{gather} \label{stopping}
 \tau:=\tau^x_R:=\inf\big\{h\geq 0 : \nnorm{X_{s+h}^x-x} > R \big\}.
 \end{gather}
 If $x\mapsto q(s, x,\xi)$ is continuous, then we have
 \begin{equation}\label{eq-prob symbol is equal to symbol of gen}
 \lim_{h\downarrow 0}E^{x,s} \left(\frac{e^{i \langle X^\tau_{s+h}-x,\xi\rangle}-1}{h}\right)= -q(s,x,\xi),
 \end{equation}
 independently of the choice of $R$ respectively $\tau$,
 that is, the probabilistic symbol of the process exists and coincides with the symbol of the generator.
\end{theorem}

\begin{proof}
As in the time homogeneous case, it is easily seen, that the operators $A_s$ fulfil the positive maximum principle. Therefore, they are pseudodifferential operators by \citet[Th\'eor\`em 1.5]{Courrege-1966}. This is the functional analytic approach; now we complement this with the probabilistic approach. Let $(\chi_n^x)_{n\in\N}$ be the sequence of cut-off functions of Lemma \ref{lem:folge} and we write $e_\xi(x):=e^{i \langle x, \xi \rangle}$ for $x,\xi\in\R^d$. By the bounded convergence theorem and formula \eqref{dynkin} we see
\begin{align*}
 E^{x,t} \left( e^{i \langle X_{t+h}^\tau -x , \xi \rangle } -1\right)
 &=\lim_{n\to\infty} \big(E^{x,t} \Bigl[\chi_n^x (X_{t+h}^\tau) e_\xi(X_{t+h}^\tau) e_{-\xi}(x) -1 \Bigr]\big) \\
 &=e_{-\xi}(x) \lim_{n\to\infty} E^{x,t} \big(\chi_n^x(X_{t+h}^\tau) e_\xi (X_{t+h}^\tau) - \chi_n^x(x) e_\xi(x) \big)\\
 &=e_{-\xi}(x) \lim_{n\to\infty} E^{x,t} \int_t^{\tau \wedge (t+h)} A_s(\chi_n^x e_\xi)(X_s) \,ds \\
 &=e_{-\xi}(x) \lim_{n\to\infty} E^{x,t} \int_t^{\tau \wedge (t+h)} A_s(\chi_n^x e_\xi)(X_{s-}) \,ds.
\end{align*}
The last equality follows since a c\`adl\`ag process has a.s.\ a
countable number of jumps and since we are integrating with respect to
Lebesgue measure. Using our Lemma \ref{lem:folge} and writing the operator $A_s$ in
integro-differential form, we obtain for all $z\in \overline{B_R(x)}$
\begin{align*}
 A_s(\chi_n & e_\xi)(z) \\
 &\leq c_\chi \Bigl(|b(s,z)| + \frac 12\sum_{j,k=1}^d |\sigma^{jk}(s,z)| + \int_{y\neq 0} (1\wedge |y|^2)\,F(s,z,dy)\Bigr) (1+|\xi|^2)\\
 &\leq c_\chi'\sup_{z\in \overline{B_R(x)}} \sup_{|\xi|\leq 1} | p(s,z,\xi)|
\end{align*}
where $c_\chi$ and $c_\chi'$ are positive constants only depending on $\chi$. 
The last estimate follows with techniques described in the appendix of \cite{Schilling-Schnurr-2010}. 
The function $p(t,x,\xi)$ is locally bounded since it is continuous.
By definition of the stopping time
$\tau$ we know that for all $t\leq s\leq\tau\wedge (t+h)$ we have $
X_{s-}\in\overline{B_R(x)}$. Therefore, the integrand $A_s(\chi_n^x
e_\xi)(X_{s-})$, $t\leq s\leq \tau\wedge (t+h)$ appearing in the above
integral is bounded and we may use the dominated convergence theorem again. 
This yields
\begin{align*}
 E^{x,t} \left( e^{i \langle X_{t+h}^\tau -x, \xi \rangle } -1\right)
 &=e_{-\xi}(x) E^{x,t} \int_t^{\tau \wedge (t+h)} \lim_{n\to \infty} A_s(\chi_n^x e_\xi)(z) |_{z=X_{s-}} \,ds \\
 &=-e_{-\xi}(x) E^{x,t} \int_t^{\tau \wedge (t+h)} e_\xi(z) p(s,z,\xi)|_{z=X_{s-}} \,ds.
\end{align*}
The second equality follows from classical results due to \citet[Sections 3.3 and 3.4]{Courrege-1966}. Therefore,
\begin{align*}
 \lefteqn{\lim_{h\downarrow 0} \frac{E^{x,t} \left(e^{i \langle X_{t+h}^\tau - x ,\xi\rangle }-1\right)}{h}}\quad \\
 &=-e_{-\xi}(x) \lim_{h \downarrow 0} E^{x,t}\left( \frac 1t \int_t^{t+h} e_\xi(X_{s-}^\tau) p(s,X_{s-}^\tau,\xi) \1_{[[ t,\tau[[}(s) \,ds\right) \\
 &=-e_{-\xi}(x) \lim_{h \downarrow 0} E^{x,t}\left( \frac 1t \int_t^{t+h} e_\xi(X_{s}^\tau) p(s,X_{s}^\tau,\xi) \1_{[[ t,\tau [[}(s) \,ds\right)
\end{align*}
since we are integrating with respect to Lebesgue measure. The process $X^\tau$ is bounded on the stochastic interval $[[ 0,\tau [[$ and $(s,x) \mapsto p(s,x,\xi)$ is continuous for every $\xi\in\R^d$. Thus, again by dominated convergence,
\begin{gather*}
 \lim_{h\downarrow 0} \frac{E^{x,t} \left(e^{i \langle X_{t+h}^\tau - x, \xi \rangle}-1\right)}{h}
 =-e_{-\xi}(x) e_{\xi}(x) p(t,x,\xi)
 =-p(t,x,\xi).
\end{gather*}
~\\[-7.5ex]
\end{proof}
\message{^^J---^^J---^^J Take care for the correct position of the end-of-proof-sign ---^^J---^^J---^^J}

\medskip

In fact, above we have been more general then needed in the context of equation \eqref{eq-levy driven processes as SDGL}. For bounded coefficients (and hence a bounded symbol) it is not necessary to introduce a stopping time. In this case the limit without stopping time, that is, 
\[
- \lim_{h\downarrow 0}E^{x,t} \left(\frac{e^{i \langle X_{t+h}-x ,\xi \rangle}-1}{t}\right)
\]
would have been sufficient and coincides with $p(t,x,\xi)$ and $q(t,x,\xi)$ above. For future reference, however, we have included the general case. Now we use the above result on L\'evy driven diffusions. 

\begin{lemma} \label{sde-symbol}
The unique strong solution of the SDE \eqref{eq-levy driven processes as SDGL} $X_t^x(\omega)$ has the symbol $q:\R_+\times\R^d\times \R^d \to \C$ given by
 \[
 q(t, x,\xi)=p(t,x,\xi)=\psi(\Phi^\top\!(x,t)\xi)
 \]
 where $\Phi$ is the coefficient of the SDE and $\psi$ the symbol of the driving L\'evy process. Hence, the (family of) generators on $\D(A_s)\supseteq C^2_{0}(\R^d)$ can be written as
\begin{equation}\label{eq-generator of a integral process}
 \begin{split}
 A_sf(x) & =  \frac{1}{2}\sum_{j,k=1}^d\bigl(\sigma\Phi(x,s)\bigr)^{j,k}\frac{\partial^2f}{\partial x_j\partial x_k}(x) + \sum_{j=1}^d\bigl(b\Phi(x,s)\bigr)^{j}\frac{\partial f}{\partial x_j}(x)\\
 & \quad{} + \int\Big(f(x + \Phi(x,s)y) - f(x) - \sum_{j=1}^d\frac{\partial f}{\partial x_j}(x)\bigr(\Phi(x,s)\chi_{\texttt C}(y)\bigr)^{j}\Big)F(dy),
 \end{split}
\end{equation}
where $F$ is the L\'{e}vy measure corresponding to $L$ which integrates $(1\wedge |y|^2)$ and $\chi_{\texttt C}$ is a cut-off function.
\end{lemma}

\begin{proof}
It is a well-known fact that the unique solution of a L\'evy driven SDE with Lipschitz coefficient is a Markov process (cf. \citet[Theorem 5.9]{Protter-1977}. In order to carry the solution, the stochastic basis on which the L\'evy process is defined has to be enlarged in a canonical way (cf. e.g. \citet[Section~V.6]{Protter-2005}). 
Next, one has to show that the solution is a universal time-inhomogeneous Markov process, that is, the transition function does not depend on the starting point. This can be done similarly to \citet[Theorem 2.47]{schnurr-2009}. The $C_0$-Feller property follows as in \citet[Theorem 2.49]{schnurr-2009}.

The calculation of the symbol works as in the proof of \citet[Theorem 3.1]{Schilling-Schnurr-2010}. We sketch the proof here in the case $d=n=1$ in order to emphasize analogies and differences: let $\tau$ be the stopping time given by
\eqref{stopping}. Fix $x,\xi \in \R$. We apply It\^o's formula for general semimartingales to the function $e_\xi(\cdot - x) =
\exp(i(\cdot-x)\xi)$:
\begin{equation}\label{threeterms}\begin{aligned}
 \frac 1h\,E^{x,t} \big( &e^{i(X^\tau_t-x)\xi} -1\big)\\
 &=\frac 1h\,E^{x,t} \bigg( \int_{t+}^{t+h} i \xi\,e^{i (X^\tau_{s-}-x) \xi} \,dX^\tau_s
 - \frac 12 \int_{t+}^{t+h} \xi^2\, e^{i (X^\tau_{s-}-x) \xi} \,d[X^\tau,X^\tau]_s^c\\
 &\quad{} +e^{-ix \xi} \sum_{t<s \leq t+h} \left(e^{i X^\tau_s \xi} - e^{i X^\tau_{s-} \xi} -i\xi
 e^{iX^\tau_{s-} \xi}\Delta X^\tau_s\right) \bigg).
\end{aligned}\end{equation}
For the first term we get {\footnotesize
\begin{align}
 \lefteqn{\frac 1h\,E^{x,t}  \int_{t+}^{t+h} \left(i \xi\, e^{i (X^\tau_{s-}-x) \xi}\right) \,dX^\tau_s}\quad \notag\\
 &= \frac 1h\,E^{x,t} \int_{t+}^{t+h} \left(i \xi \, e^{i (X^\tau_{s-}-x) \xi}\right) d\left(\int_0^s\Phi(X_{r-},r-)\1_{[[ 0,\tau ]]}(\cdot,r) \,dZ_r\right)\notag\\
 &= \frac 1h\, E^{x,t} \int_{t+}^{t+h} \left(i \xi \, e^{i (X^\tau_{s-}-x) \xi}
 \Phi(X_{s-},s-)\1_{[[ 0,\tau ]]}(\cdot,s) \right)\,dZ_s\notag\\
 &= \frac 1h\, E^{x,t} \int_{t+}^{t+h} \left(i \xi \, e^{i (X^\tau_{s-}-x) \xi}
 \Phi(X_{s-},s-)\1_{[[ 0,\tau ]]}(\cdot,s)\right) \,d(b s)\label{drift}\\
 &\quad {}+ \frac 1h\, E^{x,t} \int_{t+}^{t+h} \left(i \xi \, e^{i (X^\tau_{s-}-x) \xi}
 \Phi(X_{s-},s-)\1_{[[ 0,\tau ]]}(\cdot,s)\right) d\left(\sum_{0<r\leq s} \Delta Z_r \1_{\{\abs{\Delta Z_r}\geq 1 \}}\right)\notag
\end{align}}%
where we have used the well-known L\'evy--It\^o decomposition of the driving process. Since the integrand is bounded, the martingale terms of the L\'evy process yield again martingales whose expected value is zero.

For the drift part \eqref{drift} we obtain 
 \[
 \frac 1h\, E^{x,t} \int_{t+}^{t+h} \left(i \xi \cdot e^{i (X^\tau_{s-}-x) \xi}\Phi(X_{s-},s-) \1_{[[ 0,\tau[[}(\cdot,s) b \right)\,ds \xrightarrow{h\downarrow 0} i \xi b \Phi(x,t),
 \] 
since 
 \begin{align*}
\lefteqn{\frac 1h\, E^{x,t}   \int_{t+}^{t+h} \left(i \xi \cdot e^{i (X^\tau_{s-}-x) \xi}\Phi(X_{s-},s-) \1_{[[ 0,\tau[[}(\cdot,s) b \right)\,ds }\quad \\
 & = i \xi b  \cdot E^{x,t} \int_{t}^{t+1} \left(e^{i (X^\tau_{sh}-x) \xi}\Phi(X_{sh},sh) \1_{[[ 0,\tau[[}(\cdot,sh)\right)\,ds 
 \end{align*}

The other parts work alike. In the end we obtain
\begin{align}
 p(t, x,\xi)
 &= -i b (\Phi(x,t)\xi) + \frac{1}{2} (\Phi(x,t)\xi)\sigma (\Phi(x,t)\xi) \notag \\
 &\quad{} -\int_{y\neq 0} \left(e^{i (\Phi(x,t)\xi)y} -1 - i (\Phi(x,t)\xi)y \cdot \1_{\{\abs{y}<1\}}(y)\right) F(dy) \notag \\
 &=\psi(\Phi(x,t)\xi)\label{eq-time dependent symbol}.
\end{align}
Note that in the multi-dimensional case the matrix $\Phi(x,t)$ has to be transposed, that is, the symbol of the solution is $\psi(\Phi^\top\!(x,t)\xi)$.
We already know that the probabilistic and the classical symbol coincide. Hence, some Fourier analysis gives us the structure of the generator.
\end{proof}

\begin{rem}
We need the boundedness of $\Phi$ in order to derive the Feller property. The calculation of the symbol works even for unbounded coefficients (e.g. Lipschitz continuous $\Phi$). 
\end{rem}

Now we can prove the comparison result for L\'evy driven diffusion processes.
We consider $\F$ to be a subset of $C^2_{0}(\R^d)$ and $(L^{(i)}_t)_{t\ge 0},\,i=1,2$ to be a L\'{e}vy process with local characteristics $(b^{(i)},\sigma^{(i)}, F^{(i)})$. 
Further, let $X$ and $Y$ be L\'{e}vy driven diffusions defined by the stochastic differential equation \eqref{eq-levy driven processes as SDGL} associated to $(\Phi^{(1)},L^{(1)})$ and $(\Phi^{(2)},L^{(2)})$ with corresponding infinitesimal generator $A_s$ and $B_s$ respectively. Denote their evolution systems by $S=(S_{s,t})_{t\ge 0}$ and $T=(T_{s,t})_{t\ge 0}$, respectively.

\begin{cor}\textbf{\emph{(Comparison of time-inhomogeneous L\'{e}vy driven} \textbf{\emph{diffusion processes)}}}\label{cor-comp of Levy driven diffusion} \newline Assume that $X_0 \stackrel{d}{=} Y_0$ and
\begin{enumerate}
 \item 
$f \in \F$ implies $s\mapsto T_{s,t}f$ is right-differentiable for $0<s<t$,
 \item
$X$ is stochastically monotone w.r.t. $\F$ i.e. $f\in \F \text{ implies } S_{s,t}f\in\F$ for all $s\le t$,
 and
 \item\raisebox{-.3ex}{%
 \parbox{\linewidth}{%
 \begin{equation}\label{eq-cor-comparison of infini gen}
 A_sf\le B_sf ~[P^{X_0}]\quad\text{ for all }f\in \F \text{ and }s\le t.
 \end{equation}}}
\end{enumerate}
Then 
\begin{equation}\label{eq-cor-comp of expectations of Levy driven diffusion processes}
 X_t\fsord Y_t, \quad t\ge0.
\end{equation}
\end{cor}

\begin{proof}
We have seen above already that for every choice of $\Phi^{(i)}$ and every L\'evy process $L^{(i)}$ the solution of \eqref{eq-levy driven processes as SDGL} is a $C_0$-Feller evolution process. Since $(C_0, \nnorm{\cdot}_\infty)$ is a Banach space we can directly use Theorem \ref{theo-comparison result}
and obtain $S_{s,t}f\le T_{s,t}f$ for all $f\in \F$, $s\le t$. By using $X_0 \stackrel{d}{=} Y_0$ this yields for $f\in\F$ and $t\ge 0$
 \begin{align*}
 Ef(X_t)& =  E_0[E[f(X_t)|X_0 = x]] \\
	 & =  E_0S_{0,t}f(x) \\
	& \le  E_0T_{0,t}f(x) \\
	 & =  E_0[E[f(Y_t)|Y_0 = x]]
	 = Ef(Y_t),
 \end{align*}
where $E_{0}$ is the expectation with respect to $P^{X_0}$.
\end{proof}

Corollary \ref{cor-comp of Levy driven diffusion} shows that ordering results for L\'{e}vy driven diffusion processes for $\F\subset C^2_0$ are implied by a stochastic monotonicity assumption of one process and comparability of the local characteristics. In Section \ref{subsec-Levy driven diffusion on L^2(nu)} we consider L\'{e}vy driven diffusion processes as inhomogeneous Markov processes on $\bar{\LLl}^2_2(\nu)$ in order to derive comparison results w.r.t.\ unbounded function classes. 
Some further applications of these ordering results derived in a similar way are listed in the following remark.

\subsection{Processes with independent increments}\label{subsec-Apps and Exs-PII}
Let $(L_t)_{t\ge 0}$ be a PII such that $E|L_t|^2 < \infty$ for all $t\in\R_+$. Its characteristic function is given by $Ee^{i\langle\xi,L_t\rangle} = \exp(\int_0^t\theta_u(i\xi)du)$ with cumulant function given in \eqref{eq-ex-cumulant fucntion of PII}. We consider the case of PII with L\'{e}vy measures $F_s^{(i)},i=1,2$ which integrates $(|y| \wedge |y|^2)$. We choose the cut-off function $\chi_{\texttt C}$ as the identity. Denote for $s\ge 0$ the local characteristics of $L^{(i)}$ by $(b_s,\sigma_s^{(i)},F^{(i)}_s)$. The corresponding infinitesimal generators are $A_s$ and $B_s$ respectively. If the assumptions of Proposition \ref{propos-parade Beispiel modified L p space is well suited for trans inv MP} are fulfilled the family $(T_{s,t})_{s\le t}$ of operators defined in \eqref{eq-rem-valueable equation for PII} is a strongly continuous ES on $\bar{\LLl}^2_{2}(\nu)$. The corresponding infinitesimal generator $(A_s)_{s\ge 0}$ is given in \eqref{eq-jump term of infini gen} defined on $\D(A_s)$. For this instance we state a comparison result for the convex order $\sord{cx}
$ which is generated by 
 \[
 \fclass{cx}^0 = \fclass{cx} \cap B_{\text b} \cap \bar{\WW}^2(\nu),
 \]
where $B_\text{b}$ is the function class of $b$-bounded function and $b$ is chosen as in Remark \ref{rem-convex generator is in modified L p space}. Similar results hold for $\sord{dcx},\sord{sm},\sord{st},\sord{ism}$ as well. 

\begin{cor}\textbf{\emph{(Comparison of PII w.r.t. \boldmath $\fclass{cx}$)}} \\
\label{cor-comp of PII w.r.t. cx}
Assume $L^{(1)}_{r}\stackrel{d}{=}L^{(2)}_{r}$ for some $r\le t$. If $(S_{s,t})$ and $(T_{s,t})$ are smooth operators and if $(S_{s,t})$ leaves $B_{\text b}$ invariant, then the comparison of the local characteristics
\begin{align}
 \sigma_s^{(1)} &\sord{psd } \sigma_s^{(2)}\label{eq-cor-comp of volas-1}\\
 \int_{\R^d}f(x)F^{(1)}_s(dx) &\le\int_{\R^d}f(x) F^{(2)}_s(dx)\label{eq-cor-comp of Levy meas-1}
\end{align}
for all $s\le t$ and all $f\in\fclass{cx}$ with $f(0)=0$ such that the integrals exist, implies that
\begin{align}\label{eq-cor-explicit comparison of PII-with different jumps}
 L^{(1)}_t\sord{cx}L^{(2)}_t,\,t\ge 0.
\end{align}
\end{cor}

\begin{proof}
Let $f\in\fclass{cx}^0$. Since $(S_{s,t})$ leaves $B_{\text b}$ invariant we obtain by Remark \ref{rem-smooth operator}-(c) and the smoothness of $(S_{s,t})$ that $S_{s,t}f \in \bar{\WW}^2(\nu)$. The invariance in $\fclass{cx}$ follows by Remark \ref{rem-translation invariant trans function}-(b). 
Similarily the smoothness of $T_{s,t}$ and Proposition \ref{prop-Diff operator is infini gen} ensure that $T_{s,t}f\in\D(B_s)$ for all $s\ge 0$ (see Remark \ref{rem-translation invariant trans function}-(c)). We have choosen above the cut-off function $\chi_{\texttt C}$ as the identity, since the Levy measures fulfil $\int_{\R^d}(|y|\wedge|y|^2)F^{(i)}_s(dy)<\infty,\, i=1,2$. Thus the jump part has the form
\begin{align*}
 \lefteqn{\int_{\R^d} f(x + y) - f(x) - \sum_{j=1}^d\frac{\partial f}{\partial x_j}(x)(y)^{(j)}F_s^{(i)}(dy)}\quad \\
  &= \int_{\R^d} f(x + y) - f(x) - \sum_{j=1}^d\frac{\partial f}{\partial x_j}(x)(y)^{(j)}\1_{\{|y|<1\}} F_s^{(i)}(dy)\\
 & \quad {}- \int_{\R^d}\sum_{j=1}^d \frac{\partial f}{\partial x_j}(x)(y)^{(j)}\1_{\{|y|\ge1\}} F_s^{(i)}(dy).
\end{align*}
From this representation we note that both terms belong to $\bar{\LLl}^2_2(\nu)$ which can be seen by using similar arguments as in the proof of Lemma \ref{lem-diff operator is OKAY}. Thus, the infinitesimal generator has the form as in \eqref{eq-jump term of infini gen}. Hence, for $f\in\fclass{cx}^0$ and $s\le t$ we obtain
\begin{align}
 (B_s - A_s)f(x) & =  \frac{1}{2}\sum_{j,k=1}^d\frac{\partial^2 f}{\partial x_j\partial x_k}(x)\Bigl({\sigma_s^{(2)}}^{j,k} - {\sigma_s^{(1)}}^{j,k}\Bigr)\label{eq-difference of diffusion terms in infinigen difference}\\
 &\quad{} +\int_{\R^d}\Bigl(f(x + y) - f(x) -\sum_{j=1}^d\frac{\partial f}{\partial x_j}(x)(y)^j\Bigr)\Bigl(F^{(2)}_s- F^{(1)}_s\Bigr)(dy).\notag
\end{align}
Due to the positive semidefiniteness of $({\sigma_s^{(2)}}^{j,k}- {\sigma_s^{(1)}}^{j,k})_{j,k\le d}$ for fixed $s\ge0$, its spectral decomposition is given by $\bigl(\sum_{i\le d}\lambda_i\text{e}_i^j\text{e}_i^k\bigr)$,
where the eigenvalues $\lambda_i$ are non-negative and $\text{e} = (\text{e}_i^1,\ldots,\text{e}_i^d)$ denote the eigenvectors. Convexity of $f$ implies 
 \[
 \frac{1}{2}\sum_{i=1}^d\lambda_i\sum_{j,k=1}^d\frac{\partial^2f}{\partial x_j\partial x_k}(x) \text{e}^j_i\text{e}^k_i\ge 0.
 \]
Moreover, the function $y \mapsto f(x + y) - f(x) - \sum_{j=1}^d\frac{\partial f}{\partial x_j}(x)y^{(j)}$ is convex, since it is a sum of a convex function and a linear function. Now using the comparison of the L\'{e}vy measures in \eqref{eq-cor-comp of Levy meas-1} we get $(B_s - A_s)f(x)\ge 0$. Thus, the assumptions of Theorem \ref{theo-comparison result} are satisfied and we obtain 
 \[
 Ef(L^{(1)}_t) = E_{r}E\bigl(f(L^{(1)}_t)|L^{(1)}_{r} = x\bigr)
 \le E_{r}E\bigl(f(L^{(2)}_t)|L^{(2)}_{r} = x\bigr)
 = Ef(L^{(2)}_t),
 \]
where $E_{r}$ is the expectation with respect to $P^{L^{(1)}_{r}}$.
\end{proof}

\begin{Remark}\label{rem-sufficient condition}
The following gives for some function classes comparison conditions of the local characteristics which imply stochastic ordering of the L\'{e}vy driven diffusion processes in the case of L\'{e}vy measures, $F^{(i)},\,i=1,2$ which integrate $(|y|\wedge|y|^2)$. 
For $d\times d$ matrices $A,\,B$ with real entries, the positive semidefinite order $A\sord{psd}B$ is defined by $x^\top(B-A)x\ge0$ for all $x\in\R^d$, where $x^\top$ is the transpose of $x$. Here, we make use of the explicit representation of the generator given in \eqref{eq-generator of a integral process}.

If ${\Phi^{(i)}}:\R^d\times\R_+\to\R^{d\times d}_+,\,i=1,2$ and ${\Phi^{(1)}} \sord{psd} {\Phi^{(2)}}$, then sufficient conditions for \eqref{eq-cor-comparison of infini gen} are:
 \begin{align*}
 &\text{\emph{\textbf{Ordering}} }& &\text{\emph{\textbf{Drift}} }& &\text{\emph{\textbf{Diffusion}} }& &\text{\emph{\textbf{Jump}} }\\
 &\fclass{st} & &b^{(1)}\le b^{(2)}& &\sigma^{(1)}= \sigma^{(2)}& & F^{(1)} =F^{(2)}\\
 &\fclass{sm} & &b^{(1)} = b^{(2)}& &\sigma^{(1)}\le_{\text{d}} \sigma^{(2)}& & F^{(1)}\sord{sm}F^{(2)}\\
 &\fclass{ism} & &b^{(1)} \le b^{(2)}& &\sigma^{(1)}\le_{\text{d}} \sigma^{(2)}& & F^{(1)}\sord{sm}F^{(2)}
 \end{align*}
 where $\sigma^{(1)}\le_{\text{d}} \sigma^{(2)}$ means pointwise ordering $\sigma^{(1)}\le \sigma^{(2)},\, {\sigma^{(1)}}^{j,j}= {\sigma^{(2)}}^{j,j},\,j\le d$. For every generating function class, the latter fact can be justified as follows: Let, for instance, $f\in \fclass{st}$. Then, the first order derivatives of $f$ are non-negative. 
Moreover, let $b^{(1)}\le b^{(2)}\,\sigma^{(1)}= \sigma^{(2)}$ and $F^{(1)}=F^{(2)}$, hence the corresponding generators in \eqref{eq-generator of a integral process} are ordered, too. Then, the comparison result follows from Corollary 
\ref{cor-comp of Levy driven diffusion on L^2(nu)}.
\end{Remark}

\subsection{L\'{e}vy driven diffusion processes on \boldmath $\bar{\LLl}^2_2(\nu)$}
\label{subsec-Levy driven diffusion on L^2(nu)}
We reconsider L\'{e}vy driven diffusion processes as inhomogeneous Markov processes. In comparison to Corollary \ref{cor-comp of Levy driven diffusion} where these processes were considered as strongly continuous evolution systems on $C_{0}(\R^d)$ we now consider them as strongly continuous evolution systems on $\bar{\LLl}^2_2(\nu)$ assuming for both processes the kernel assumption $(K)$ as well as $(C)$ and $(D)$. We state a comparison result for L\'{e}vy driven diffusion w.r.t. $\le_{\F^0}$, where $\F$ is from \eqref{eq-fst def} - \eqref{eq-mixes def} and
	\[
	\F^0 = \F \cap B_{\text b} \cap \bar{\WW}^2(\nu).
	\]
Now let $(L^{(i)}_t)_{t\ge 0},\,i=1,2$ be a L\'{e}vy process such that $E|L_t|^2<\infty,\,t>0$, with local characteristics $(b^{(i)},\sigma^{(i)}, F^{(i)})$ such that $F^{(i)},\,i=1,2$ integrates $(|y|\wedge |y|^2)$. Further, let $X$ and $Y$ be L\'{e}vy driven diffusions defined by the stochastic differential equation \eqref{eq-levy driven processes as SDGL} associated to $(\Phi^{(1)},L^{(1)})$ and $(\Phi^{(2)},L^{(2)})$ respectively. 
Denote their infinitesimal generator by $A_s$ and $B_s$ respectively defined in \eqref{eq-generator of a integral process} and their evolution systems on $\bar{\LLl}^2_2(\nu)$ by $S=(S_{s,t})_{t\ge 0}$ and $T=(T_{s,t})_{t\ge 0}$.

Extending the $ES$ of $X$ and $Y$ onto $\bar{\LLl}^2_2(\nu)$ by Proposition \ref{propos-trans func of general MP is evo system on L^2(nu)} it is clear that the form and the mapping behaviour of the associated infinitesimal generators (see equation \eqref{eq-generator of a integral process}) for $f \in \bar{\WW}^2(\nu)$ is deduced similarly as in Lemma \ref{lem-diff operator is OKAY} and Proposition \ref{prop-Diff operator is infini gen}.

\begin{cor}[Comparison of L\'{e}vy driven diffusion processes]%
\label{cor-comp of Levy driven diffusion on L^2(nu)}\hfill\\
 Assume that $X$ and $Y$ possess the transition kernels $(P^X_t)_{t\ge0}$ and $(P^Y_t)_{t\ge0}$, which satisfy $(K)$, $(C)$ and $(D)$. 
If $X_0 \stackrel{d}{=} Y_0$ and
\begin{enumerate}
 \item $f \in \F$ implies $s\mapsto T_{s,t}f$ is right-differentiable for $0<s<t$,
 \item$X$ is stochastically monotone w.r.t. $\F^0$, and
 \item\raisebox{-.3ex}{%
 \parbox{\linewidth}{%
 \begin{equation}\label{eq-cor-comparison of infini gen on H s}
 A_sf\le B_sf\quad\text{ a.s. for all }f\in \F \text{ and }s\le t,
 \end{equation}}}
\end{enumerate}
then 
\begin{equation}\label{eq-cor-comp of expectations of Levy driven diffusion processes on H s}
 X_t\fsord Y_t,\,t\ge0.
\end{equation}
\end{cor}

\setcounter{equation}{0}
\section{Pseudo-differential operators on Sobolev Slobodeckii spaces}\label{sec-general remark}
Finally, we make some remarks on a method to compare general Markov processes based on \emph{Sobolev Slobodeckii} spaces and the theory of pseudo-differential operators. We relate this method to strongly continuous evolution systems as used in this paper. 
The main reference for this final section is \cite{Boettcher-2008} who gave a representation result for pseudo-differential operators on Sobolev Slobodeckii spaces similar as in Theorem \ref{theo-evolution problem}. 
His result is based on developments on pseudo-differential operators in \cite{Eskin-1981}, \cite{Hoh-1998} and \cite{Jacob-2001,Jacob-2002}. 

For $r\in\R$ the \emph{Sobolev-Slobodeckii }space is defined as
\begin{equation}\label{eq-Sobolev-Slobo Space}
 \HH^r(\R^d) := \bigl\{u\in \Ss'(\R^d)\big| \hat u \in\LLl^1_{\text{loc}}(\R^d)\text{ s.t. }\|u\|_r<\infty\bigr\}
\end{equation}
with 
\begin{equation}\label{eq-Sobolov-Slobo Space Norm}
 \|u\|_r^2 =\int_{\R^d}|\hat u(\xi)|^2(1+|\xi|)^{2r}d\xi,
\end{equation}
where $\Ss'(\R^d)$ is the dual space of the \emph{Schwartz} space $\Ss(\R^d)$ and $\hat u$ denotes the Fourier transform of $u$. The space $\HH^r(\R^d)$ endowed with the scalar product $\langle\cdot,\cdot\rangle_r$,
 \[
\langle u,v \rangle_r:=\int\hat u(\xi)\overline{\hat v(\xi)}(1+|\xi|)^{2r} d\xi,\,u,v\in\HH^r(\R^d)
\]
is a separable Hilbert space (see \citet[Theorem 4.1]{Eskin-1981}). Moreover for $r\ge 0\,(\HH^r(\R^d),\langle\cdot,\cdot\rangle_r)$ is isomorph to 
 \[
\Bigl\{u\in\LLl^2(\R^d,\C)\Big|\int|\hat u(\xi)|^2(1+|\xi|)^{2r}d\xi < \infty\Bigr\}
\]
endowed with the same scalar product. 
The space $\HH^r(\R^d)$ is an interesting example of a Banach function class which uses regularity conditions on the elements but allows unbounded functions. 

Elements $u$ of $\HH^r(\R^d)$ fulfil certain integrability and regularity assumptions since the existence of higher moments of $\hat u$ in \eqref{eq-Sobolov-Slobo Space Norm} corresponds to higher order regularity of the function. For instance, $u\in\LLl^1(\R)$ is continuously differentiable if the first moment of $\hat{u}$ is integrable, that is if $\int|\xi\|\hat u(\xi)| d\xi<\infty$, and in this case we have
\begin{align*}
 \frac{\partial}{\partial x}u(x) & =  \frac{\partial}{\partial x}\frac{1}{(2\pi)^{1/2}}\int e^{-ix\xi}\hat u(\xi)d\xi\\
 & =  \frac{1}{(2\pi)^{1/2}}\int (-i)e^{-ix\xi}\xi\hat u(\xi)d\xi.
\end{align*}
A pseudo-differential operator $q(t,x,D)$ is defined on a suitable space by 
\begin{equation}
 q(t,x,D)f(x):=(2\pi)^{-d/2}\int_{\R^d}e^{-i\langle x,\xi \rangle} q(t,x,\xi)\hat f(\xi)d\xi 
\end{equation}
if the right-hand side exists. The function $q$ which defines the operator is called the \emph{symbol} of the pseudo-differential operator.
 
In \citet[Theorem 2.3 ]{Boettcher-2008} a representation result for solutions of evolution equations on $\HH^r(\R^d)$ is established. For an application of this result it has to be checked that the transition functions $T_{s,t}$ and the infinitesimal generator $A_s$ (corresponding to a time-inhomogeneous Markov process) are pseudo-differential operators with symbols in suitable 'symbol classes', i.e. the symbols have sufficient regularity and growth behaviour. 
Under this condition on the symbol it is verified in Theorem 1.8 in \cite{Boettcher-2008} that the transition function $T_{s,t}$ maps $\HH^r(\R^d)$ into itself. This mapping property is based on general theory of pseudo-differential operators as described in \cite{Eskin-1981}, \cite{Hoh-1998} and \cite{Jacob-2001,Jacob-2002}. 

Our representation result in Theorem \ref{theo-evolution problem} can be applied to this context if the transition operator $T_{s,t}$ is a strongly continuous ES on $\HH^r(\R^d)$, that is, it maps $\HH^r(\R^d)$ into $\HH^r(\R^d)$. In other words, it is sufficient to prove an analogous result to Proposition \ref{propos-parade Beispiel for evo systems} for $\HH^r(\R^d)$. 

We can apply our representation Theorem \ref{theo-evolution problem} and get as corollary a result related to Böttchers representation result without any growth and regularity conditions on the symbol of the infinitesimal generator.

\begin{cor}\label{cor-weak evolution system on H}
 Under the assumptions of Theorem \ref{theo-evolution problem}, let $(T_{s,t})_{s\le t}$ be a strong\-ly continuous $ES$ on $\HH^r(\R^d)$, then the solutions $F_t$ of the weak evolution problem \eqref{eq-initial value problem general} on $\HH^r(\R^d)$ have the representation
\begin{equation}\label{eq-cor-representation of weak solution on H}
 F_t(s) = T_{s,t}F_t(t) - \int_s^t T_{s,r}G(r) dr. 
\end{equation}
\end{cor}

For the application of Corollary \ref{cor-weak evolution system on H} first note that any transition operator $(T_{s,t})$ of a Markov process $(X_t)_{t\ge0}$ has a representation as a pseudo-differential operator.

For the proof let $f\in\HH^r(\R^d),\,r\ge0$ and consider a time-inhomogeneous Markov process $X=(X_t)_{t\ge0}$ with transition function $(P_{s,t})_{s\le t}$. Denote by $F(f)$ the Fourier transform of $f$ and by $F^{-1}$ the inverse Fourier transform. Then
\begin{align}
 T_{s,t}f(x)&= E_{s,x}\Bigl(F^{-1}\bigl(F(f)\bigr)(X_t)\Bigr)\notag\\
 &= E_{s,x}\biggl(\frac{1}{(2\pi)^{d/2}}\int_{\R^d}e^{-i\langle X_t,\xi\rangle}F(f)(\xi)d\xi\biggl)\notag\\
 &= \frac{1}{(2\pi)^{d/2}}\int_{\R^d}e^{-i\langle x,\xi\rangle}q_s(t,x,\xi)\hat f(\xi)d\xi,\label{eq-evolution system of general MP is a PDO}
\end{align}
where $q_s(t,x,\xi):=E\bigl(e^{-i\langle X_t -x,\xi\rangle}|X_s =x\bigr)$. Thus $(T_{s,t})$ is a pseudo-differential operator with symbol $q_s(t,x,\xi)$.

For applications of Corollary \ref{cor-weak evolution system on H} to comparison results we have to establish that $T_{s,t}:\HH^r(\R^d) \to \HH^r(\R^d)$. This can be done in the case of PII, which are particularly suitable for the pseudo-differential operator approach. 
\begin{example}\textbf{\emph{(Transition functions of PII on $\HH^r(\R^d)$)}}\\ 
\label{trans func of PII are pdo s}
Let $X = (X_t)_{t\ge 0}$ be PII with characteristic function 
$E\bigl(e^{-i\langle X_t -X_s,\xi\rangle}\bigr)=\exp(\int_s^t\theta_u(-i\xi)du)=:\hat{\mu}_{s,t}(-\xi).$
Then from the representation in equation \eqref{eq-evolution system of general MP is a PDO} we obtain for $f\in\HH^r(\R^d)$, that
 \begin{equation}\label{eq-ex-evolution system of a PII is a PDO}
 T_{s,t}f(x) = \frac{1}{(2\pi)^{d/2}}\int_{\R^d}e^{-i\langle x,\xi\rangle}\hat{\mu}_{s,t}(-\xi)\hat f(\xi)d\xi.
 \end{equation}
Hence the characteristic function $\hat{\mu}_{s,t}(-\xi)$ is the symbol of $T_{s,t}$. It is not difficult to see that the transition operator of a PII is a pseudo-differential operator with bounded $C^\infty$-symbol, if all absolute moments of the L\'{e}vy measure exists. 
Thus the symbol lies in the symbol class $S^0_0$. In consequence the conditions on the symbol for the mapping property $T_{s,t}:\HH^r(\R^d) \to \HH^r(\R^d)$ in Theorem 1.8 of \cite{Boettcher-2008} are fulfilled and Corollary \ref{cor-weak evolution system on H} gives the necessary representation result. 
result Corollary \ref{cor-comp of PII w.r.t. cx}.

As consequence of Corollary \ref{cor-weak evolution system on H} the comparison theorem (Theorem \ref{theo-comparison result}) allows us to state a comparison result for PII for function classes $\F\subset \HH^r(\R^d)$. Note however that by this approach, based on $\HH^r(\R^d)$, we need the strong condition of existence of all absolute moments of the L\'{e}vy measure, while our comparison result based on modified $\bar{\LLl}^2_2(\nu)$ spaces as in Corollary \ref{cor-comp of PII w.r.t. cx} is established under much weaker conditions.
\end{example}

For general Markov processes with symbol $q_s(t,x,\xi)$ dependent on $x$ it is not easy to check that its symbol has suitable regularity and growth behaviour. Hence this method is difficult to apply for time-inhomogeneous Markov processes as for example for time-inhomogeneous diffusions. Thus it seems that the approach in this paper based on our representation result for strongly continuous evolution system in general Banach spaces is more flexible and easier to apply in examples.

\paragraph*{Acknowledgement:} AS would like to acknowledge financial support by the DFG (German Science Foundation) for the project SCHN 1231/2-1.
\bibliographystyle{plainnat}
{
\bibliography{CompMP}
\label{sec:bib}
}

\paragraph*{Address:}Department of Mathematical Stochastics, University of Freiburg,\\e-mail: ruschen@stochastik.uni-freiburg.de;
FAX: +49 761 203 5661; Phone: +49 761 203 5665

\noindent and

Faculty of Mathematics, TU Dortmund \\ e-mail: aschnurr@math.tu-dortmund.de, FAX: +49 231 755 3064 ; Phone: +49 231 755 3099 

\end{document}